\documentclass[11pt,fleqn]{article}

\usepackage{graphicx}
\usepackage{amsthm,amsmath}
\usepackage{amssymb,dsfont,mathrsfs}

\usepackage{color}

\renewcommand{\baselinestretch}{1.05} 

\numberwithin{equation}{section}
\newtheorem{theorem}{Theorem}[section]

\newtheorem{proposition}{Proposition}[section]

\newtheorem{remark}{Remark}[section]
\newtheorem{assumption}{Assumption}[section]

\def\ba{\boldsymbol{a}}

\def\bc{\boldsymbol{c}}
\def\be{\boldsymbol{e}}

\def\bg{\boldsymbol{g}}

\def\bi{\boldsymbol{i}}

\def\bk{\boldsymbol{k}}
\def\bl{\boldsymbol{l}}

\def\bu{\boldsymbol{u}}
\def\bv{\boldsymbol{v}}

\def\bY{\boldsymbol{Y}}

\def\bphi{\boldsymbol{\phi}}
\def\bvarphi{\boldsymbol{\varphi}}
\def\bpi{\boldsymbol{\pi}}
\def\bnu{\boldsymbol{\nu}}

\def\bzero{\mathbf{0}}
\def\bone{\mathbf{1}}

\def\calS{\mathcal{S}}

\def\calL{\mathcal{L}}

\def\scrI{\mathscr{I}}

\def\spr{\mbox{\rm spr}}
\def\diag{\mbox{\rm diag}}
\def\adj{\mbox{\rm adj}}

\def\vecM{\mbox{\rm vec}}

\title{Exact asymptotic formulae of the stationary distribution of a discrete-time 2d-QBD process: an example and additional proofs}
\author{Toshihisa Ozawa${}^{\dagger}$ and Masahiro Kobayashi${}^{\dagger\dagger}$ \\ 
${}^{\dagger}$Faculty of Business Administration, Komazawa University \\
${}^{\dagger\dagger}$Department of Mathematical Science, Tokai University \\
${}^{\dagger}$1-23-1 Komazawa, Setagaya-ku, Tokyo 154-8525, Japan \\
E-mail: ${}^{\dagger}$toshi@komazawa-u.ac.jp, ${}^{\dagger\dagger}$m\_kobayashi@tsc.u-tokai.ac.jp 
}
\date{}

\begin{document}

\maketitle

\begin{abstract}
A discrete-time two-dimensional quasi-birth-and-death process (2d-QBD process), denoted by $\{\bY_n\}=\{(X_{1,n},X_{2,n},J_n)\}$, is a two-dimensional skip-free random walk $\{(X_{1,n},X_{2,n})\}$ on $\mathbb{Z}_+^2$ with a supplemental process $\{J_n\}$ on a finite set $S_0$.  The supplemental process $\{J_n\}$ is called a phase process. The 2d-QBD process $\{\bY_n\}$ is a Markov chain in which the transition probabilities of the two-dimensional process $\{(X_{1,n},X_{2,n})\}$ vary according to the state of the phase process $\{J_n\}$. This modulation is assumed to be space homogeneous except for the boundaries of $\mathbb{Z}_+^2$. 
Under certain conditions, the directional exact asymptotic formulae of the stationary distribution of the 2d-QBD process have been obtained in Ref.\ \cite{Ozawa18a}. 
In this paper, we give an example of 2d-QBD process and proofs of some lemmas and propositions appeared in Ref.\ \cite{Ozawa18a}.

\smallskip
{\it Key wards}: quasi-birth-and-death process, stationary distribution, asymptotic property, matrix analytic method, two-dimensional reflecting random walk

\end{abstract}

%
%

%
\section{Introduction} \label{sec:intro}

A discrete-time two-dimensional quasi-birth-and-death process (2d-QBD process), denoted by $\{\bY_n\}=\{(X_{1,n},X_{2,n},J_n)\}$, is a two-dimensional skip-free random walk $\{(X_{1,n},X_{2,n})\}$ on $\mathbb{Z}_+^2$ with a supplemental process $\{J_n\}$ on a finite set $S_0$ (see Ref.\ \cite{Ozawa13}).  The supplemental process $\{J_n\}$ is called a phase process. The 2d-QBD process $\{\bY_n\}$ is a Markov chain in which the transition probabilities of the two-dimensional process $\{(X_{1,n},X_{2,n})\}$ vary according to the state of the phase process $\{J_n\}$. This modulation is assumed to be space homogeneous except for the boundaries of $\mathbb{Z}_+^2$. 
Assume that the 2d-QBD process $\{\bY_n\}$ is irreducible, aperiodic and positive recurrent, and denote by $\bnu=(\bnu_{k,l},\,(k,l)\in\mathbb{Z}_+^2)$ its stationary distribution, where $\bnu_{k,l}=(\nu_{k,l,j},\,j\in S_0)$ and $\nu_{k,l,j}=\lim_{n\to\infty} \mathbb{P}(\bY_n=(k,l,j))$.  
In Ref.\ \cite{Ozawa18a},  we have obtained, under certain conditions, the directional exact asymptotic formulae $h_1(k)$ and $h_2(k)$ that satisfy, for some nonzero vector $\bc_1$ and $\bc_2$, 
\begin{equation}
\lim_{k\to\infty} \frac{\bnu_{k,0}}{h_1(k)}=\bc_1,\qquad 
\lim_{k\to\infty} \frac{\bnu_{0,k}}{h_2(k)}=\bc_2.
\end{equation}
In this paper, we give an example of 2d-QBD process and proofs of some lemmas and propositions appeared in Ref.\ \cite{Ozawa18a}, in order to make that paper easy to understand.

The rest of the paper is organized as follows. In Section \ref{sec:model}, we summarize main results of Ref.\ \cite{Ozawa18a} and related topics. Numerical examples are presented in Section \ref{sec:example}, where a single-server two-queue model is considered. Proofs of some lemmas and propositions in Ref.\ \cite{Ozawa18a} are given in Section \ref{sec:proofs}.

\bigskip
\noindent\textit{Notation.}\quad
%
A set $\mathbb{H}$ is defined as $\mathbb{H}=\{-1,0,1\}$ and $\mathbb{H}_+$ as $\mathbb{H}_+=\{0,1\}$. 
For $a,b\in\mathbb{R}_+$, $\mathbb{C}[a,b]$ and $\mathbb{C}[a,b)$ are defined as $\mathbb{C}[a,b] = \{z\in\mathbb{C}: a\le |z|\le b \}$ and $\mathbb{C}[a,b) = \{z\in\mathbb{C}: a\le |z|< b \}$, respectively. $\mathbb{C}(a,b]$ and $\mathbb{C}(a,b)$ are analogously defined. 
%
%
For $r>0$, $\varepsilon>0$ and $\theta\in[0,\pi/2)$, $\tilde{\Delta}_r(\varepsilon,\theta)$ is defined as  
\[
\tilde{\Delta}_r(\varepsilon,\theta) = \{z\in\mathbb{C}: |z|<r+\varepsilon,\ z\ne r,\ |\arg(z-r)|>\theta \}. 
\]
For $r>0$, we denote by ``$\tilde{\Delta}_r\ni z\to r$" that $\tilde{\Delta}_r(\varepsilon,\theta)\ni z\to r$ for some $\varepsilon>0$ and some $\theta\in [0,\pi/2)$. 
%
For a matrix $A=(a_{ij})$, we denote by $[A]_{ij}$ the $(i,j)$-entry of $A$. The transpose of $A$ is denoted by $A^\top$. 
We denote by $\spr(A)$ the spectral radius of $A$, which is the maximum modulus of eigenvalue of $A$. 
We denote by $|A|$ the matrix each of whose entries is the modulus of the corresponding entry of $A$, i.e., $|A|=(|a_{ij}|)$. 
$O$ is a matrix of $0$'s, $\bone$ is a column vector of $1$'s and $\bzero$ is a column vector of $0$'s; their dimensions are determined in context. $I$ is the identity matrix. 
For a $k\times l$ matrix $A=\begin{pmatrix} \ba_1 & \ba_2 & \cdots & \ba_l \end{pmatrix}$, ${\rm vec}(A)$ is a $kl\times 1$ vector defined as 
\[
\vecM(A) = \begin{pmatrix} \ba_1 \cr \ba_2 \cr \vdots \cr \ba_l \end{pmatrix}. 
\]
For matrices $A$, $B$ and $C$, the identity $\vecM(ABC)=(C^\top\otimes A)\,\vecM(B)$ holds (see, for example, 
Horn and Johnson \cite{Horn91}).

%
%
\section{Model description and main results of Ref.\ \cite{Ozawa18a}} \label{sec:model} 

Here we summarize Section 2 of Ref.\ \cite{Ozawa18a}. 
%
%
\subsection{2d-QBD process} \label{sec:tdQBD}

%
Let $S_0=\{1,2,...,s_0\}$ be a finite set, where $s_0$ is the number of elements of $S_0$. A 2d-QBD process $\{\bY_n\} = \{ (X_{1,n}, X_{2,n}, J_n) \}$ is a discrete-time Markov chain on the state space $\calS = \mathbb{Z}_+^2 \times S_0$. The transition probability matrix $P$ of $\{\bY_n\}$ is represented in block form as
\[
P=\begin{pmatrix} P_{(x_1,x_2),(x_1',x_2')};(x_1,x_2),(x_1',x_2')\in\mathbb{Z}_+^2 \end{pmatrix},
\]
where each block $P_{(x_1,x_2),(x_1',x_2')}$ is given as $P_{(x_1,x_2),(x_1',x_2')}=\begin{pmatrix} p_{(x_1,x_2,j),(x_1',x_2',j')}; j,j'\in S_0 \end{pmatrix}$ and for $(x_1,x_2,j),(x_1',x_2',j')\in\calS$, $p_{(x_1,x_2,j),(x_1',x_2',j')} = \mathbb{P}(\bY_1=(x_1',x_2',j')\,|\,\bY_0=(x_1,x_2,j))$.
The block matrices are given in terms of $s_0\times s_0$ non-negative matrices 
\[
A_{i,j}, i,j\in\mathbb{H},\quad
A^{(1)}_{i,j}, i\in\mathbb{H},j\in\mathbb{H}_+,\quad
A^{(2)}_{i,j}, i\in\mathbb{H}_+,j\in\mathbb{H},\quad
A^{(0)}_{i,j}, i,j\in\mathbb{H}_+,
\]
as follows: for $(x_1,x_2),(x_1',x_2')\in\mathbb{Z}_+^2$, 
\[
P_{(x_1,x_2),(x_1',x_2')} = \left\{ \begin{array}{ll} 
A_{\varDelta x_1,\varDelta x_2}, & \mbox{if  $x_1\ne 0$, $x_2\ne 0$, $\varDelta x_1,\varDelta x_2\in\mathbb{H}$}, \cr 
A^{(1)}_{\varDelta x_1,\varDelta x_2}, & \mbox{if $x_1\ne 0$, $x_2=0$, $\varDelta x_1\in\mathbb{H}$, $\varDelta x_2\in\mathbb{H}_+$}, \cr
A^{(2)}_{\varDelta x_1,\varDelta x_2}, & \mbox{if $x_1=0$, $x_2\ne 0$, $\varDelta x_1\in\mathbb{H}_+$, $\varDelta x_2\in\mathbb{H}$}, \cr
A^{(0)}_{\varDelta x_1,\varDelta x_2}, & \mbox{if $x_1=x_2=0$, $\varDelta x_1,\varDelta x_2\in\mathbb{H}_+$}, \cr
O, & \mbox{otherwise}, \end{array} \right. 
\]
where $\varDelta x_1=x_1'-x_1$ and $\varDelta x_2=x_2'-x_2$. 
%
Define matrices $A_{*,*}$, $A^{(1)}_{*,*}$, $A^{(2)}_{*,*}$ and $A^{(0)}_{*,*}$ as
\[
A_{*,*} = \sum_{i,j\in\mathbb{H}} A_{i,j},\quad
A^{(1)}_{*,*} = \sum_{i\in\mathbb{H},j\in\mathbb{H}_+} A^{(1)}_{i,j},\quad
A^{(2)}_{*,*} = \sum_{i\in\mathbb{H}_+,j\in\mathbb{H}} A^{(2)}_{i,j},\quad
A^{(0)}_{*,*} = \sum_{i,j\in\mathbb{H}_+} A^{(0)}_{i,j}. 
\]
All of these matrices are stochastic. 
%
%
%
%
We assume the following condition. 
\begin{assumption} \label{as:irreducibleYn}
The Markov chain $\{\bY_n\}$ is irreducible and aperiodic.
\end{assumption}


We consider three kinds of Markov chain generated from $\{\bY_n\}$ by removing one or two boundaries and denote them by $\{\tilde{\bY}_n\}=\{(\tilde{X}_{1,n},\tilde{X}_{2,n},\tilde{J}_n)\}$, $\{\tilde{\bY}^{(1)}_n\}=\{(\tilde{X}^{(1)}_{1,n},\tilde{X}^{(1)}_{2,n},\tilde{J}^{(1)}_n)\}$ and $\{\tilde{\bY}^{(2)}_n\}=\{(\tilde{X}^{(2)}_{1,n},\tilde{X}^{(2)}_{2,n},\tilde{J}^{(2)}_n)\}$, respectively.  
$\{\tilde{\bY}_n\}$ is a Markov chain on the state space $\mathbb{Z}^2\times S_0$ and it is generated from $\{\bY_n\}$ by removing the boundaries on the $x_1$ and $x_2$-axes. 
For $(x_1,x_2,j),(x_1',x_2',j')\in\mathbb{Z}^2\times S_0$, denote by $\tilde{p}_{(x_1,x_2,j),(x_1',x_2',j')}$ the transition probability $\mathbb{P}(\tilde{\bY}_1=(x_1',x_2',j')\,|\,\tilde{\bY}_0=(x_1,x_2,j))$. 
The transition probability matrix $\tilde{P}$ of $\{\tilde{\bY}_n\}$ is represented in block form as
\[
\tilde{P}=\begin{pmatrix} \tilde{P}_{(x_1,x_2),(x_1',x_2')};(x_1,x_2),(x_1',x_2')\in\mathbb{Z}^2 \end{pmatrix},
\]
where $\tilde{P}_{(x_1,x_2),(x_1',x_2')}=\begin{pmatrix} \tilde{p}_{(x_1,x_2,j),(x_1',x_2',j')}; j,j'\in S_0 \end{pmatrix}$; each block $\tilde{P}_{(x_1,x_2),(x_1',x_2')}$ is given as 
\[
\tilde{P}_{(x_1,x_2),(x_1',x_2')}
= \left\{ \begin{array}{ll} A_{\varDelta x_1,\varDelta x_2}, & \mbox{if $\varDelta x_1,\varDelta x_2\in\mathbb{H}$}, \cr 
O, & \mbox{otherwise}, \end{array} \right. 
\]
where $\varDelta x_1=x_1'-x_1$ and $\varDelta x_2=x_2'-x_2$. 
%
The Markov chain $\{\tilde{\bY}_n\}$ is a Markov chain generated by $\{A_{k,l}:\,k,l\in\mathbb{H}\}$. 
%
%
%
$\{\tilde{\bY}^{(1)}_n\}$ is a Markov chain on the state space $\mathbb{Z}\times\mathbb{Z}_+\times S_0$ and it is generated from $\{\bY_n\}$ by removing the boundary on the $x_2$-axes. 
For $(x_1,x_2,j),(x_1',x_2',j')\in\mathbb{Z}\times\mathbb{Z}_+\times S_0$, denote by $\tilde{p}^{(1)}_{(x_1,x_2,j),(x_1',x_2',j')}$ the transition probability $\mathbb{P}(\tilde{\bY}^{(1)}_1=(x_1',x_2',j')\,|\,\tilde{\bY}^{(1)}_0=(x_1,x_2,j))$. 
The transition probability matrix $\tilde{P}^{(1)}$ of $\{\tilde{\bY}^{(1)}_n\}$ is represented in block form as
\[
\tilde{P}^{(1)} = \begin{pmatrix} \tilde{P}^{(1)}_{(x_1,x_2),(x_1',x_2')};(x_1,x_2),(x_1',x_2')\in\mathbb{Z}\times\mathbb{Z}_+ \end{pmatrix},
\]
where $\tilde{P}^{(1)}_{(x_1,x_2),(x_1',x_2')}=\begin{pmatrix} \tilde{p}^{(1)}_{(x_1,x_2,j),(x_1',x_2',j')}; j,j'\in S_0 \end{pmatrix}$; each block $\tilde{P}^{(1)}_{(x_1,x_2),(x_1',x_2')}$ is given as 
\[
\tilde{P}^{(1)}_{(x_1,x_2),(x_1',x_2')} 
= \left\{ \begin{array}{ll} 
A_{\varDelta x_1,\varDelta x_2}, & \mbox{if  $x_2\ne 0$, $\varDelta x_1,\varDelta x_2\in\mathbb{H}$}, \cr 
A^{(1)}_{\varDelta x_1,\varDelta x_2}, & \mbox{if $x_2=0$, $\varDelta x_1\in\mathbb{H}$, $\varDelta x_2\in\mathbb{H}_+$}, \cr
O, & \mbox{otherwise}, 
\end{array} \right. 
\]
where $\varDelta x_1=x_1'-x_1$ and $\varDelta x_2=x_2'-x_2$. 
$\{\tilde{\bY}^{(2)}_n\}$ is a Markov chain on the state space $\mathbb{Z}_+\times\mathbb{Z}\times S_0$ and it is generated from $\{\bY_n\}$ by removing the boundary on the $x_1$-axes. The transition probability matrix $\tilde{P}^{(2)}$ of $\{\tilde{\bY}^{(2)}_n\}$ is analogously given. 
%
The Markov chain $\{\tilde{\bY}^{(1)}_n\}$ is a Markov chain generated by $\{ \{A_{k,l}:k,l\in\mathbb{H}\},\,\{A_{k,l}^{(1)}:k\in\mathbb{H},\,l\in\mathbb{H}_+\} \}$ and $\{\tilde{\bY}^{(2)}_n\}$ is that generated by $\{ \{A_{k,l}:k,l\in\mathbb{H}\},\,\{A_{k,l}^{(2)}:k\in\mathbb{H}_+,\,l\in\mathbb{H}\} \}$. 
%
%
%
Hereafter, we assume the following condition.
%
\begin{assumption} \label{as:Akl_irreducible} 
$\{\tilde{\bY}_n\}$, $\{\tilde{\bY}^{(1)}_n\}$ and $\{\tilde{\bY}^{(2)}_n\}$ are irreducible and aperiodic. 
\end{assumption}

%
Under Assumption \ref{as:Akl_irreducible}, $A_{*,*}$ is irreducible and aperiodic.
%
$A_{*,*}$ is, therefore, positive recurrent and ergodic since its dimension is finite. We denote by $\bpi_{*,*}$ the stationary distribution of $A_{*,*}$.

\medskip
%

%
A 2d-QBD process $\{\bY_n\} = \{ (X_{1,n}, X_{2,n}, J_n) \}$ can be represented as a QBD process with a countable phase space in two ways: One is $\{\bY_n^{(1)}\} = \{(X_{1,n}, (X_{2,n}, J_n))\}$, where $X_{1,n}$ is the level and $(X_{2,n}, J_n)$ the phase, and the other $\{\bY_n^{(2)}\} = \{(X_{2,n}, (X_{1,n}, J_n))\}$, where $X_{2,n}$ is the level and $(X_{1,n}, J_n)$ the phase (see Subsection 2.3 of Ref.\ \cite{Ozawa18a}).
Let $R^{(1)}$ and $R^{(2)}$ be the rate matrices of the QBD processes $\{\bY^{(1)}\}$ and $\{\bY^{(2)}\}$, respectively. 
Hereafter, we assume the following condition.
\begin{assumption} \label{as:irreducibleR1R2}
The rate matrices $R^{(1)}$ and $R^{(2)}$ are irreducible.
\end{assumption}


%
%
%

%
\subsection{Stationary condition} \label{sec:stationary_cond}

We define induced Markov chains and the mean drift vectors derived from the induced Markov chains (see Ref.\ \cite{Fayolle95}).
Since the 2d-QBD process is a kind of two-dimensional reflecting random walk, there exist three induced Markov chains: $\calL^{\{1,2\}}$, $\calL^{\{1\}}$ and $\calL^{\{2\}}$. 
$\calL^{\{1,2\}}$ is the phase process $\{\tilde{J}_n\}$ of $\{\tilde{\bY}_n\}$ and it is a Markov chain governed by the transition probability matrix $A_{*,*}$. The mean drift vector $\ba^{\{1,2\}}=(a^{\{1,2\}}_1,a^{\{1,2\}}_2)$ derived from $\calL^{\{1,2\}}$ is given as 
\[
a^{\{1,2\}}_1 = \bpi_{*,*} (-A_{-1,*}+A_{1,*}) \bone,\quad 
a^{\{1,2\}}_2 = \bpi_{*,*} (-A_{*,-1}+A_{*,1}) \bone,  
\]
where for $k\in\mathbb{H}$, $A_{k,*}=\sum_{l\in\mathbb{H}}A_{k,l}$ and $A_{*,k}=\sum_{l\in\mathbb{H}}A_{l,k}$. 
Define block tri-diagonal transition probability matrices $A^{(1)}_*$ and $A^{(2)}_*$ as 
\[
A^{(1)}_* = 
\begin{pmatrix}
A^{(1)}_{*,0} & A^{(1)}_{*,1} & & & \cr
A_{*,-1} & A_{*,0} & A_{*,1} & & \cr
& A_{*,-1} & A_{*,0} & A_{*,1} & \cr
& & \ddots & \ddots & \ddots 
\end{pmatrix},\ 
A^{(2)}_* = 
\begin{pmatrix}
A^{(2)}_{0,*} & A^{(2)}_{1,*} & & & \cr
A_{-1,*} & A_{0,*} & A_{1,*} & & \cr
& A_{-1,*} & A_{0,*} & A_{1,*} & \cr
& & \ddots & \ddots & \ddots 
\end{pmatrix}, 
\]
where for $k\in\mathbb{H}_+$, $A^{(1)}_{*,k}=\sum_{l\in\mathbb{H}}A^{(1)}_{l,k}$ and $A^{(2)}_{k,*}=\sum_{l\in\mathbb{H}}A^{(2)}_{k,l}$. 
Under Assumption \ref{as:Akl_irreducible}, $A^{(1)}_*$ and $A^{(2)}_*$ are irreducible and aperiodic.
%
%
$\calL^{\{1\}}$ (resp.\ $\calL^{\{2\}}$) is a partial process $\{(\tilde{X}^{(1)}_{2,n},\tilde{J}^{(1)}_n)\}$ of $\{\tilde{\bY}^{(1)}_n\}$ (resp.\ $\{(\tilde{X}^{(2)}_{1,n},\tilde{J}^{(2)}_n)\}$ of $\{\tilde{\bY}^{(2)}_n\}$) and it is a Markov chain governed by $A^{(1)}_*$ (resp.\ $A^{(2)}_*$). 
Since $\calL^{\{1\}}$ (resp.\ $\calL^{\{2\}}$) is one-dimensional QBD process, it is positive recurrent if and only if $a^{\{1,2\}}_2<0$ (resp.\ $a^{\{1,2\}}_1<0$). 
Denote by $\bpi^{(1)}_*=(\bpi^{(1)}_{*,k},k\in\mathbb{Z}_+)$ and $\bpi^{(2)}_*=(\bpi^{(2)}_{*,k},k\in\mathbb{Z}_+)$ the stationary distributions of $\calL^{\{1\}}$ and $\calL^{\{2\}}$, respectively, if they exist. Then, the mean drift vectors $\ba^{\{1\}}=(a^{\{1\}}_1,a^{\{1\}}_2)$ and $\ba^{\{2\}}=(a^{\{2\}}_1,a^{\{2\}}_2)$ derived from $\calL^{\{1\}}$ and $\calL^{\{2\}}$ are given as 
\begin{align*}
a^{\{1\}}_1 &= \bpi_{*,0}^{(1)}\big(-A_{-1,0}^{(1)}-A_{-1,1}^{(1)}+A_{1,0}^{(1)}+A_{1,1}^{(1)} \big)\,\bone + \bpi_{*,1}^{(1)}\big(I-R^{(1)}_*\big)^{-1} \big(-A_{-1,*}+A_{1,*} \big)\,\bone,\\
a^{\{1\}}_2 &= 0,\quad 
a^{\{2\}}_1 = 0, \\
a^{\{2\}}_2 &= \bpi_{*,0}^{(2)}\big(-A_{0,-1}^{(2)}-A_{1,-1}^{(2)}+A_{0,1}^{(2)}+A_{1,1}^{(2)} \big)\,\bone + \bpi_{*,1}^{(2)}\big(I-R^{(2)}_*\big)^{-1} \big(-A_{*,-1}+A_{*,1} \big)\,\bone, 
\end{align*}
where $R^{(1)}_*$ and $R^{(2)}_*$ are the rate matrices of $A_*^{(1)}$ and $A_*^{(2)}$, respectively. 
The condition ensuring  $\{\bY_n\}$ is positive recurrent or transient is given by Lemma 2.1 of Ref.\ \cite{Ozawa18a}. 
%
%
%
%
%
In order for the 2d-QBD process $\{\bY_n\}$ to be positive recurrent, we assume the following condition. 
\begin{assumption} \label{as:stability_cond} 
$a^{\{1,2\}}_1$ or $a^{\{1,2\}}_2$ is negative. 
If $a^{\{1,2\}}_1<0$ and $a^{\{1,2\}}_2<0$, then $a^{\{1\}}_1<0$ and $a^{\{2\}}_2<0$; if $a^{\{1,2\}}_1>0$ and $a^{\{1,2\}}_2<0$, then $a^{\{1\}}_1<0$; if $a^{\{1,2\}}_1<0$ and $a^{\{1,2\}}_2>0$, then $a^{\{2\}}_2<0$. 
\end{assumption}
Denote by $\bnu = \left(\nu_{k,l,j},\, (k,l,j)\in\mathbb{Z}_+^2\times S_0 \right)$ the stationary distribution of $\{\bY_n\}$. $\bnu$ is represented in block form as $\bnu = \left( \bnu_{k,l},\, (k,l)\in\mathbb{Z}_+^2 \right)$, where $\bnu_{k,l} = \left(\nu_{k,l,j},\, j\in S_0 \right)$.

%
\subsection{Directional geometric decay rates}

%
Let $z_1$ and $z_2$ be positive real numbers and define a matrix function $C(z_1,z_2)$ as
\[
C(z_1,z_2) 
= \sum_{i,j\in\mathbb{H}} A_{i,j} z_1^i z_2^j 
= \sum_{j\in\mathbb{H}} A_{*,j}(z_1) z_2^j 
= \sum_{i\in\mathbb{H}} A_{i,*}(z_2) z_1^i, 
\] 
where 
\[
A_{*,j}(z_1)=\sum_{i\in\mathbb{H}} A_{i,j} z_1^i,\quad 
A_{i,*}(z_2)=\sum_{j\in\mathbb{H}} A_{i,j} z_2^j. 
\]
The matrix function $C(z_1,z_2)$ is nonnegative and, under Assumption \ref{as:Akl_irreducible}, it is irreducible and aperiodic. Let $\chi(z_1,z_2)$ be the Perron-Frobenius eigenvalue of $C(z_1,z_2)$, i.e., $\chi(z_1,z_2)=\spr(C(z_1,z_2))$, and let $\bu^C(z_1,z_2)$ and $\bv^C(z_1,z_2)$ be the Perron-Frobenius left and right eigenvectors satisfying $\bu^C(z_1,z_2)\bv^C(z_1,z_2)=1$. 
We have $\bu^C(1,1)=\bpi_{*,*}$ and $\bv^C(1,1)=\bone$. 
%
%
%
Define a closed set $\bar{\Gamma}$ as 
\[
\bar{\Gamma}=\{ (s_1,s_2)\in\mathbb{R}^2 : \chi(e^{s_1},e^{s_2}) \le 1\}.
\]
Since $\chi(1,1)=\chi(e^0,e^0)=1$, $\bar{\Gamma}$ contains the point of $(0,0)$ and thus it is not empty. By Proposition 2.3 and Lemma 2.2 of Ref.\ \cite{Ozawa18a}, $\bar{\Gamma}$ is a bounded convex set. 
%
%
%
%
%
%
%
%
Furthermore, by Lemma 2.3 of Ref.\ \cite{Ozawa18a}, the closed set $\bar{\Gamma}$ is not a singleton. 
%
%
%
For $i\in\{1,2\}$, define the lower and upper extreme values of $\bar{\Gamma}$ with respect to $s_i$, denoted by $\underline{s}_i^*$ and $\bar{s}_i^*$, as 
\[
\underline{s}_i^* = \min_{(s_1,s_2)\in\bar{\Gamma}} s_i,\quad 
\bar{s}_i^* = \max_{(s_1,s_2)\in\bar{\Gamma}} s_i,
\]
where $-\infty< \underline{s}_i^* \le 0$ and $0\le \bar{s}_i^* < \infty$. 
For $i\in\{1,2\}$, define $\underline{z}_i^*$ and $\bar{z}_i^*$ as $\underline{z}_i^*=e^{\underline{s}_i^*}$ and $\bar{z}_i^*=e^{\bar{s}_i^*}$, respectively, where $0<\underline{z}_i^*\le 1$ and $1\le \bar{z}_i^*<\infty$. 
%
%
%
By Proposition 2.4 of Ref.\ \cite{Ozawa18a}, for each $z_1\in(\underline{z}_1^*,\bar{z}_1^*)$ (resp.\ $z_2\in(\underline{z}_2^*,\bar{z}_2^*)$), equation $\chi(z_1,z_2)=1$ has just two different real solutions $\underline{\zeta}_2(z_1)$ and $\bar{\zeta}_2(z_1)$ (resp.\ $\underline{\zeta}_1(z_2)$ and $\bar{\zeta}_1(z_2)$), where $\underline{\zeta}_2(z_1)<\bar{\zeta}_2(z_1)$ (resp.\ $\underline{\zeta}_1(z_2)<\bar{\zeta}_1(z_2)$). 
For $z_1=\underline{z}_1^*\mbox{ or }\bar{z}_1^*$ (resp.\ $z_2=\underline{z}_2^*\mbox{ or }\bar{z}_2^*$), it has just one real solution $\underline{\zeta}_2(z_1)=\bar{\zeta}_2(z_1)$ (resp.\ $\underline{\zeta}_1(z_2)=\bar{\zeta}_1(z_2)$). 
If $z_1\notin[\underline{z}_1^*,\bar{z}_1^*]$ (resp.\ $z_2\notin[\underline{z}_2^*,\bar{z}_2^*]$), it has no real solutions. 

Consider the following matrix quadratic equations of $X$:
\begin{align}
A_{*,-1}(z_1) + A_{*,0}(z_1) X + A_{*,1}(z_1) X^2 = X, 
\label{eq:equationG1} \\
A_{-1,*}(z_2) + A_{0,*}(z_2) X + A_{1,*}(z_2) X^2 = X.  
\label{eq:equationG2}
\end{align}
Denote by $G_1(z_1)$ and $G_2(z_2)$ the minimum nonnegative solutions to matrix equations (\ref{eq:equationG1}) and (\ref{eq:equationG2}), respectively. 
Furthermore, consider the following matrix quadratic equations of $X$:
\begin{align}
 X^2 A_{*,-1}(z_1) + X A_{*,0}(z_1) + A_{*,1}(z_1) = X, 
\label{eq:equationR1} \\
 X^2 A_{-1,*}(z_2) + X A_{0,*}(z_2) + A_{1,*}(z_2) = X.  
\label{eq:equationR2}
\end{align}
Denote by $R_1(z_1)$ and $R_2(z_2)$ the minimum nonnegative solutions to matrix equations (\ref{eq:equationR1}) and (\ref{eq:equationR2}), respectively.
%
%
By Lemma 2.4 of Ref.\ \cite{Ozawa18a}, for $z\in\mathbb{R}_+\setminus\{0\}$, the minimum nonnegative solutions $G_1(z)$ and $R_1(z)$ (resp.\ $G_2(z)$ and $R_2(z)$) to matrix equations (\ref{eq:equationG1}) and (\ref{eq:equationR1}) (resp.\ equations (\ref{eq:equationG2}) and (\ref{eq:equationR2})) exist if and only if $z\in[\underline{z}_1^*,\bar{z}_1^*]$ (resp.\ $z\in[\underline{z}_2^*,\bar{z}_2^*])$. 
%
%
Furthermore, by Proposition 2.5 of Ref.\ \cite{Ozawa18a}, we have, for $z_1\in[\underline{z}_1^*,\bar{z}_1^*]$ and $z_2\in[\underline{z}_2^*,\bar{z}_2^*]$, 
\begin{align}
&\spr(G_1(z_1)) = \underline{\zeta}_2(z_1),\quad \spr(R_1(z_1)) = \bar{\zeta}_2(z_1)^{-1}, \\
&\spr(G_2(z_2)) = \underline{\zeta}_1(z_2),\quad \spr(R_2(z_2)) = \bar{\zeta}_1(z_2)^{-1}. 
\end{align}
%
Define matrix functions $C_1(z_1,X)$ and $C_2(X,z_2)$ as 
\begin{equation}
C_1(z_1,X) = A_{*,0}^{(1)}(z_1) + A_{*,1}^{(1)}(z_1) X,\quad 
C_2(X,z_2) = A_{0,*}^{(2)}(z_2) + A_{1,*}^{(2)}(z_2) X, 
\end{equation}
where $X$ is an $s_0\times s_0$ matrix and, for $i\in\mathbb{H}_+$,
\[
A_{*,i}^{(1)}(z_1) = \sum_{j\in\mathbb{H}} A_{j,i}^{(1)} z_1^j,\quad 
A_{i,*}^{(2)}(z_2) = \sum_{j\in\mathbb{H}} A_{i,j}^{(2)} z_2^j.
\]
Define $\psi_1(z_1)$ and $\psi_2(z_2)$ as
\[
\psi_1(z_1) = \spr(C_1(z_1,G_1(z_1))),\quad 
\psi_2(z_2) = \spr(C_2(G_2(z_2),z_2)). 
\]
%

\begin{figure}[htbp]
\begin{center}
\includegraphics[width=16cm,trim=60 320 0 50]{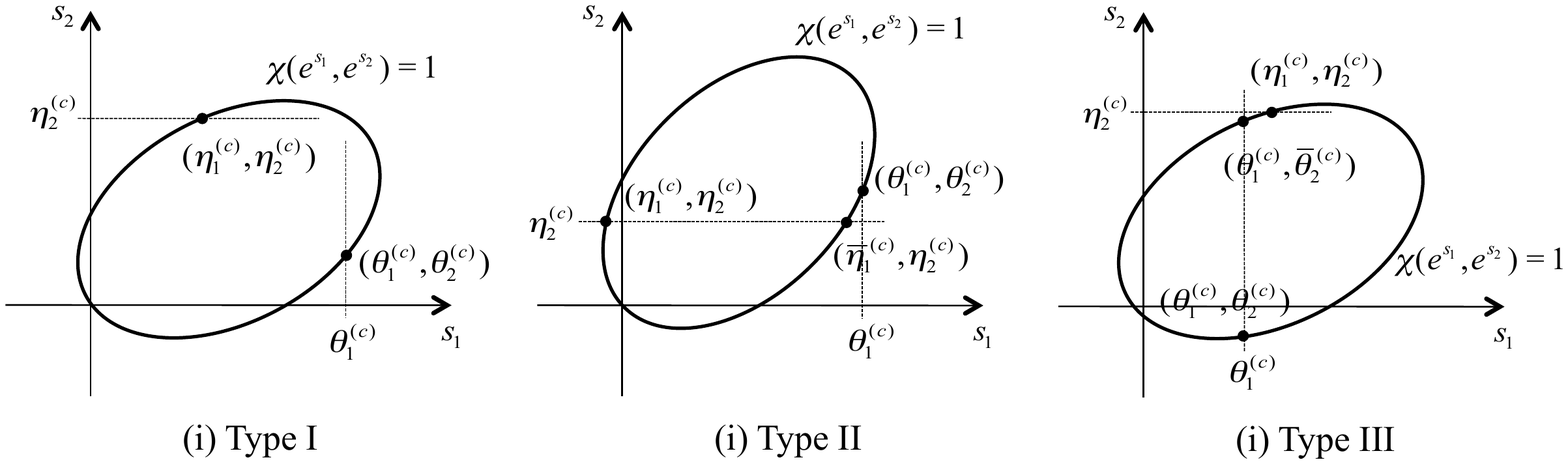} 
\caption{Configuration of $(\theta_1^{(c)},\theta_2^{(c)})$ and $(\eta_1^{(c)},\eta_2^{(c)})$}
\label{fig:configuration}
\end{center}
\end{figure}
%
Next, we introduce points $(\theta_1^{(c)},\theta_2^{(c)})$ and $(\eta_1^{(c)},\eta_2^{(c)})$ on the closed curve $\chi(e^{s_1},e^{s_2})=1$, as follows (see Fig.\ \ref{fig:configuration}):
\begin{align*}
\theta_1^{(c)} = \max\{s_1\in[0,\bar{s}_1^*]: \psi_1(e^{s_1})\le 1\},\quad 
\theta_2^{(c)} = \log \underline{\zeta}_2(e^{\theta_1^{(c)}}), \cr
\eta_2^{(c)} = \max\{s_2\in[0,\bar{s}_2^*]: \psi_2(e^{s_2})\le 1\},\quad 
\eta_1^{(c)} = \log \underline{\zeta}_1(e^{\eta_2^{(c)}}).
\end{align*}
Furthermore, define $\bar{\theta}_2^{(c)}$ and $\bar{\eta}_1^{(c)}$ as
\[
\bar{\theta}_2^{(c)} = \log \bar{\zeta}_2(e^{\theta_1^{(c)}}),\quad
\bar{\eta}_1^{(c)} = \log \bar{\zeta}_1(e^{\eta_2^{(c)}}).
\]
%
%
%
%
By Lemma 2.5 of Ref.\ \cite{Ozawa18a}, we see that $\theta_1^{(c)}$ and $\eta_2^{(c)}$ are always positive. 
%
Since $\bar{\Gamma}$ is convex, we can classify the possible configuration of points $(\theta_1^{(c)},\theta_2^{(c)})$ and $(\eta_1^{(c)},\eta_2^{(c)})$ in the following manner (see Fig.\ \ref{fig:configuration}). 
\begin{align*}
&\mbox{Type I: $\eta_1^{(c)}<\theta_1^{(c)}$ and $\theta_2^{(c)}<\eta_2^{(c)}$},\quad 
\mbox{Type II: $\eta_1^{(c)}<\theta_1^{(c)}$ and $\eta_2^{(c)}\le\theta_2^{(c)}$}, \\
&\mbox{Type III: $\theta_1^{(c)}\le\eta_1^{(c)}$ and $\theta_2^{(c)}<\eta_2^{(c)}$}. 
\end{align*}
Define the directional decay rates, $\xi_1$ and $\xi_2$, of the stationary distribution of the 2d-QBD process as
\begin{align*}
\xi_1 &= - \lim_{n\to\infty} \frac{1}{n} \log\nu_{n,j,k}\quad\mbox{for any $j$ and $k$},\\
\xi_2 &= - \lim_{n\to\infty} \frac{1}{n} \log\nu_{i,n,k}\quad\mbox{for any $i$ and $k$}. 
\end{align*}
By Lemma 2.6 of Ref.\ \cite{Ozawa18a}, $\xi_1$ and $\xi_2$ are given as follows.
%
%
\begin{align}
&(\xi_1,\xi_2) = \left\{ 
\begin{array}{ll}
(\theta_1^{(c)},\eta_2^{(c)}), & \mbox{Type I}, \cr
(\bar{\eta}_1^{(c)},\eta_2^{(c)}), & \mbox{Type II}, \cr
(\theta_1^{(c)},\bar{\theta}_2^{(c)}), & \mbox{Type III}.
\end{array} \right.
\end{align}
%
The directional geometric decay rates $r_1$ in $x_1$-coordinate and $r_2$ in $x_2$-coordinate are given as $r_1=e^{\xi_1}$ and $r_2=e^{\xi_2}$.

%
%
\subsection{Main results of Ref.\ \cite{Ozawa18a}}

We assume the following technical condition. 
\begin{assumption} \label{as:G1_eigen_z1max}
All the eigenvalues of $G_1(r_1)$ are distinct. Also, those of $G_2(r_2)$ are distinct.
\end{assumption}

%
%
%

The exact asymptotic formulae $h_1(k)$ in $x_1$-coordinate and $h_2(k)$ in $x_2$-coordinate 
%
%
are given as follows. 
\begin{theorem}[Theorem 2.1 of Ref.\ \cite{Ozawa18a}] \label{th:mainresults}
Under Assumptions \ref{as:irreducibleYn} through \ref{as:G1_eigen_z1max}, in the case of Type I, the exact asymptotic formulae are given as
\begin{align}
h_1(k) 
&= \left\{ \begin{array}{ll}
r_1^{-k}, & \psi_1(\bar{z}_1^*)>1, \cr
k^{-\frac{1}{2}(2 l_1-1)} (\bar{z}_1^*)^{-k}, & \psi_1(\bar{z}_1^*)=1, \cr
k^{-\frac{1}{2}(2 l_1+1)} (\bar{z}_1^*)^{-k}, & \psi_1(\bar{z}_1^*)<1,
\end{array} \right. \quad
%
h_2(k) 
= \left\{ \begin{array}{ll}
r_2^{-k}, & \psi_2(\bar{z}_2^*)>1, \cr
k^{-\frac{1}{2}(2 l_2-1)} (\bar{z}_2^*)^{-k}, & \psi_2(\bar{z}_2^*)=1, \cr
k^{-\frac{1}{2}(2 l_2+1)} (\bar{z}_2^*)^{-k}, & \psi_2(\bar{z}_2^*)<1,
\end{array} \right. 
\end{align}
where $l_1$ and $l_2$ are some positive integers. In the case of Type II, they are given as 
\begin{align}
h_1(k) 
&= \left\{ \begin{array}{ll}
r_1^{-k}, & \eta_2^{(c)}<\theta_2^{(c)}, \cr
k\,r_1^{-k}, & \eta_2^{(c)}=\theta_2^{(c)}\ \mbox{and}\ \psi_1(\bar{z}_1^*)>1, \cr
(\bar{z}_1^*)^{-k}, & \eta_2^{(c)}=\theta_2^{(c)}\ \mbox{and}\ \psi_1(\bar{z}_1^*)=1, \cr
k^{-\frac{1}{2}} (\bar{z}_1^*)^{-k}, & \eta_2^{(c)}=\theta_2^{(c)}\ \mbox{and}\ \psi_1(\bar{z}_1^*)<1,
\end{array} \right. \quad
%
h_2(k) =r_2^{-k}. 
\end{align}
In the case of Type III, they are given as 
\begin{align}
h_1(k) &=r_1^{-k}, \quad
%
h_2(k) 
= \left\{ \begin{array}{ll}
r_2^{-k}, & \theta_1^{(c)}<\eta_1^{(c)}, \cr
k\,r_2^{-k}, & \theta_1^{(c)}=\eta_1^{(c)}\ \mbox{and}\ \psi_2(\bar{z}_2^*)>1, \cr
(\bar{z}_2^*)^{-k}, & \theta_1^{(c)}=\eta_1^{(c)}\ \mbox{and}\ \psi_2(\bar{z}_2^*)=1, \cr
k^{-\frac{1}{2}} (\bar{z}_2^*)^{-k}, & \theta_1^{(c)}=\eta_1^{(c)}\ \mbox{and}\ \psi_2(\bar{z}_2^*)<1. 
\end{array} \right. \quad
\end{align}
\end{theorem}


%

%
%
\section{An example} \label{sec:example}

We consider the same queueing model as that used in Ozawa \cite{Ozawa13}. 
It is a single-server two-queue model in which the server visits the queues alternatively, serves one queue (denoted by Q$_1$) according to a 1-limited service and the other queue (denoted by Q$_2$) according to an exhaustive-type $K$-limited service (see Fig.\ \ref{fig:two_queue}). 
Customers in class 1 arrive at Q$_1$ according to a Poisson process with intensity $\lambda_1$, those in class 2 arrive at Q$_2$ according to another Poisson process with intensity $\lambda_2$, and they are mutually independent. 
Service times of class-1 customers are exponentially distributed with mean $1/\mu_1$, those of class-2 customers are also exponentially distributed with mean $1/\mu_2$, and they are mutually independent. 
The arrival processes and service times are also mutually independent. We define $\rho_i,\,i=1,2$, $\lambda$ and $\rho$ as $\rho_i=\lambda_i/\mu_i$, $\lambda=\lambda_1+\lambda_2$ and $\rho=\rho_1+\rho_2$, respectively. 
We refer to this model as a $(1,K)$-limited service model.

\begin{figure}[htbp]
\begin{center}
\includegraphics[width=16cm,trim=0 400 0 30]{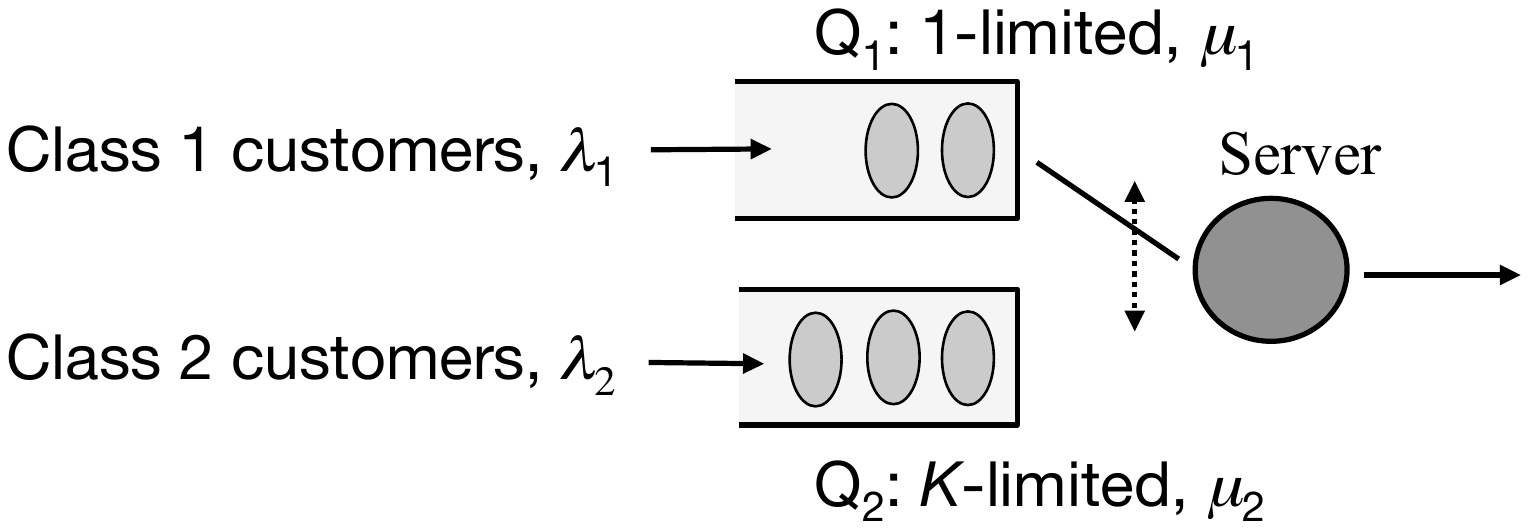} 
\caption{Single server two-queue model}
\label{fig:two_queue}
\end{center}
\end{figure}

Let $X_1(t)$ be the number of customers at Q$_1$ at time $t$, $X_2(t)$ that of customers at Q$_2$ and $J(t) \in S_0=\{0, 1, ..., K\}$ the server state. 
When $X_1(t)=X_2(t)=0$, $J(t)$ takes one of the states in $S_0$ at random; 
When $X_1(t)\ge 1$ and $X_2(t)=0$, it also takes one of the states in $S_0$ at random; 
When $X_1(t)=0$ and $X_2(t)\ge 1$, it takes the state of $0$ or $1$ at random if the server is serving the $K$-th customer in class 2 during a visit of the server at Q$_2$ and takes the state of $j\in\{2, ..., K\}$ if the server is serving the $(K-j+1)$-th customer in class 2; 
When $X_1(t)\ge 1$ and $X_2(t)\ge 1$, it takes the state of $0$ if the server is serving a class-1 customer and takes the state of $j\in\{1, ..., K\}$ if the server is serving the $(K-j+1)$-th customer in class 2 during a visit of the server at Q$_2$. 
The process $\{(X_1(t), X_2(t), J(t))\}$ becomes a continuous-time 2d-QBD process on the state space $\calS=\mathbb{Z}_+^2\times S_0$. By uniformization with parameter $\nu=\lambda+\mu_1+\mu_2$, we obtain the corresponding discrete-time 2d-QBD process, $\{(X_{1,n}, X_{2,n}, J_n)\}$, governed by the following $(K+1)\times(K+1)$ block matrices:
\begin{align*}
&A_{1,0} = \frac{\lambda_1}{\nu} I,\quad 
A_{0,1} = \frac{\lambda_2}{\nu} I,\quad 
A_{1,1} = A_{1,-1} = A_{-1,1} = A_{-1,-1} = O, \\
&A_{-1,0} = \frac{\mu_1}{\nu} \begin{pmatrix}
0 & \cdots & 0 & 1 \cr
0 & \cdots & 0 & 0 \cr
\vdots & \ddots &  \vdots & \vdots \cr
0 & \cdots & 0 & 0 \cr \end{pmatrix},\quad
A_{0,-1} = \frac{\mu_2}{\nu} \begin{pmatrix}
0 & & & \cr
1 & 0 & & \cr
 & \ddots & \ddots & \cr
 & & 1 & 0 \end{pmatrix}, \\
&A_{0,0} = I + \frac{1}{\nu} \begin{pmatrix}
-(\lambda+\mu_1) & & & \cr
 & -(\lambda+\mu_2) & & \cr
 & & \ddots & \cr
 & & & -(\lambda+\mu_2) \end{pmatrix}, 
\end{align*}
\begin{align*}
&A^{(1)}_{1,0} = \frac{\lambda_1}{\nu} I,\quad 
A^{(1)}_{0,0} = \left(1-\frac{(\lambda+\mu_1)}{\nu}\right) I,\quad 
A^{(1)}_{1,1} = A^{(1)}_{-1,1} = O, \\
&A^{(1)}_{0,1} = \frac{\lambda_2}{\nu} \begin{pmatrix}
1 & 0 & \cdots & 0\cr
1 & 0 & \cdots & 0\cr
\vdots & \vdots & \ddots & \vdots \cr
1 & 0 & \cdots & 0 \end{pmatrix},\quad
A^{(1)}_{-1,0} = \frac{\mu_1}{\nu(K+1)} \begin{pmatrix}
1 & \cdots & 1\cr
\vdots & \ddots & \vdots \cr
1 & \cdots & 1 \end{pmatrix}, 
\end{align*}
\begin{align*}
&A^{(2)}_{0,1} = \frac{\lambda_2}{\nu} I,\quad 
A^{(2)}_{0,0} = \left(1-\frac{(\lambda+\mu_2)}{\nu}\right) I,\quad 
A^{(2)}_{1,1} = A^{(2)}_{1,-1} = O, \\
&A^{(2)}_{1,0} = \frac{\lambda_1}{\nu} \begin{pmatrix}
0 & 1 & 0 & & \cr
0 & 1 & 0 & & \cr
0 & 0 & 1 & & \cr
 & & \ddots & \ddots & \cr
 & & & 0 & 1 \end{pmatrix},\quad 
A^{(2)}_{0,-1} = \frac{\mu_2}{\nu} \begin{pmatrix}
0 & 0 & 0 & \cdots & 0 & 1 \cr
0 & 0 & 0 &  \cdots & 0 & 1 \cr
1/2 & 1/2 & 0 & \cdots & 0 & 0 \cr
0 & 0 & 1 & \cdots & 0 & 0 \cr
& & & \ddots & \ddots & & \cr
0 & 0 & 0 & \cdots & 1 & 0  \end{pmatrix}, 
\end{align*}
\begin{align*}
&A^{(0)}_{1,0} = \frac{\lambda_1}{\nu} I,\quad 
A^{(0)}_{1,1} = O,\quad
A^{(0)}_{0,0} = \left(1-\frac{\lambda}{\nu}\right) I,\quad 
A^{(0)}_{0,1} = \frac{\lambda_2}{\nu} \begin{pmatrix}
0 & \cdots & 0 & 1\cr
\vdots & \ddots & \vdots & \vdots \cr
0 & \cdots & 0 & 1  \cr
0 & \cdots & 0 & 1 \end{pmatrix}. 
\end{align*}

Here we note that, because of the form of $A_{-1,0}$, only the $(K+1)$-th column of $G_2(w)$ is nonzero for all $w\in[\underline{z}_2^*,\bar{z}_2^*]$ and the value of $0$ is always the eigenvalue of $G_2(w)$ with algebraic multiplicity $K$. Hence, $G_2(r_2)$ does not satisfy Assumption \ref{as:G1_eigen_z1max} but, in this case, we can give $G_2(w)$ in a specific form. 
%
$G_2(w)$ as well as $G_1(z)$ can analogously be defined for a complex variable (see Section 4 of Ref.\ \cite{Ozawa18a}). For $z,w\in\mathbb{C}$, define a matrix function $L(z,w)$ as 
\begin{equation}
L(z,w) = z w (C(z,w)-I) = z A_{*,-1}(z)+z(A_{*,0}(z)-I) w+z A_{*,1}(z) w^2, 
\label{eq:Lzw_definition}
\end{equation}
and denote by $\phi(z,w)$ the determinant of $L(z,w)$, i.e., $\phi(z,w)=\det L(z,w)$. $L(z,w)$ is a $(K+1)\times(K+1)$ matrix polynomial in $w$ with degree $2$ and every entry of $L(z,w)$ is a polynomial in $z$ and $w$. 
$\phi(z,w)$ is also a polynomial in $z$ and $w$. 
%
For $w\in\mathbb{C}[\underline{z}_2^*,\bar{z}_2^*]$, let $\alpha(w)$ be the solution to $\phi(z,w)=0$ that corresponds to $\underline{\zeta}_1(w)$ when $w\in[\underline{z}_2^*,\bar{z}_2^*]$. 
Let $\bv(w)$ be the column vector satisfying $L(\alpha(w),w) \bv(w) = \bzero$. This $\bv(w)$ is the right eigenvector of $G_2(w)$ with respect to the eigenvalue $\alpha(w)$. If $[\bv(w)]_{K+1}\ne 0$, $G_2(w)$ is given as
\begin{equation}
G_2(w) = \begin{pmatrix}  \bzero & \cdots & \bzero & \frac{\alpha(w)}{[\bv(w)]_{K+1}}\,\bv(w) \end{pmatrix}, 
\end{equation}
and it is decomposed as
\begin{equation}
G_2(w) = V_2(w) J_2(w) V_2(w)^{-1}, 
\end{equation}
where 
\begin{align}
&J_2(w) = \diag(0,\dots,0,\alpha(w)),\\
&V_2(w) = \begin{pmatrix}  \be_1 & \be_2 & \cdots & \be_K & \frac{1}{[\bv(w)]_{K+1}}\,\bv(w) \end{pmatrix}, \\
&V_2(w)^{-1}=\begin{pmatrix}  \be_1 & \be_2 & \cdots & \be_K & 2 \be_{K+1}-\frac{1}{[\bv(w)]_{K+1}}\,\bv(w) \end{pmatrix} 
\end{align}
and, for $i\in\{1,2,...,K+1\}$, $\be_i$ is the $i$-th unit vector of dimension $K+1$. 
For any $w'\in\mathbb{C}$, if $\alpha(w)$ is analytic at $w=w'$, $G_2(w)$ is also entry-wise analytic at $w=w'$. Hence, the results of Lemma 4.7 of Ref.\ \cite{Ozawa18a} 
hold for this $G_2(w)$. 
Furthermore, because of the decomposition of $G_2(w)$ above, the results in Section 5 of Ref.\ \cite{Ozawa18a} 
also hold for $\bvarphi_2(w)=\sum_{j=1}^\infty \bnu_{0,j} w^j$. 
Therefore, in the $(1,K)$-limited service model, the results of Theorem  \ref{th:mainresults} hold for the exact asymptotic formula $h_2(k)$.

In Tables \ref{tab:example1} and \ref{tab:example2}, we show numerical examples of the geometric decay rates $r_1$ and $r_2$ in the $(1,K)$-limited service model, where the value of $K$ is varied. 
We have numerically verified that $G_1(r_1)$ satisfies Assumption \ref{as:G1_eigen_z1max} for all the numerical examples. 
In both the tables, the type of configuration of points $(\theta_1^{(c)},\theta_2^{(c)})$ and $(\eta_1^{(c)},\eta_2^{(c)})$ is Type I for all the values of $K$. We see from the tables that if $K\ge 6$, the exact asymptotic function $h_1(k)$ is geometric and if $K< 6$, it is not. 
On the other hand, in the case of Table \ref{tab:example1}, the exact asymptotic function $h_2(k)$ is not geometric for any value of $K$; in the case of Table \ref{tab:example2}, if $K\le 2$, $h_2(k)$ is geometric and $K>2$, it is not.

%
\begin{table}[htbp]
\caption{Directional geometric asymptotic rates ($\lambda_1=0.3$, $\lambda_2=0.3$, $\mu_1=1$, $\mu_2=1$, $\rho=0.6$)}
\begin{center}
\begin{tabular}{c|c|cc|cc|cc}
$K$ & Type & $a_1^{\{1,2\}}$ & $a_2^{\{1,2\}}$ & $\psi(\bar{z}_1^*)-1$ & $\psi(\bar{z}_2^*)-1$ & $r_1$ & $r_2$ \cr \hline
1 & I & $-0.0769$ & $-0.0769$ & $-$ & $-$ & $1.968$ & $1.968$ \cr
2 & I & $-0.0128$ & $-0.141$ & $-$ & $-$ & $1.706$ & $2.698$ \cr
3 & I & $0.0192$ & $-0.173$ & $-$ & $-$ & $1.674$ & $3.171$ \cr
4 & I & $0.0385$ & $-0.192$ & $-$ & $-$ & $1.668$ & $3.464$ \cr
5 & I & $0.0513$ & $-0.205$ & $-$ & $-$ & $1.667$ & $3.659$ \cr
6 & I & $0.0604$ & $-0.214$ & $+$ & $-$ & $1.667$ & $3.797$ \cr
7 & I & $0.0673$ & $-0.221$ & $+$ & $-$ & $1.667$ & $3.899$ \cr
8 & I & $0.0726$ & $-0.226$ & $+$ & $-$ & $1.667$ & $3.978$ \cr
10 & I & $0.0804$ & $-0.234$ & $+$ & $-$ & $1.667$ & $4.091$ \cr
20 & I & $0.0971$ & $-0.251$ & $+$ & $-$ & $1.667$ & $4.308$ \cr
50 & I & $0.108$ & $-0.262$ & $+$ & $-$ & $1.667$ & $4.421$
\end{tabular}
\end{center}
\label{tab:example1}
\end{table}%

\begin{table}[htbp]
\caption{Directional geometric asymptotic rates ($\lambda_1=0.24$, $\lambda_2=0.7$, $\mu_1=1.2$, $\mu_2=1$, $\rho=0.9$)}
\begin{center}
\begin{tabular}{c|c|cc|cc|cc}
$K$ & Type & $a_1^{\{1,2\}}$ & $a_2^{\{1,2\}}$ & $\psi(\bar{z}_1^*)-1$ & $\psi(\bar{z}_2^*)-1$ & $r_1$ & $r_2$ \cr \hline
1 & I & $-0.0973$ & $0.0492$ & $-$ & $+$ & $3.646$ & $1.116$ \cr
2 & I & $-0.0360$ & $-0.00187$ & $-$ & $+$ & $1.844$ & $1.116$ \cr
3 & I & $-0.00665$ & $-00263$ & $-$ & $-$ & $1.183$ & $1.135$ \cr
4 & I & $0.0105$ & $-0.0406$ & $-$ & $-$ & $1.102$ & $1.273$ \cr
5 & I & $0.0218$ & $-0.0500$ & $-$ & $-$ & $1.095$ & $1.392$ \cr
6 & I & $0.0298$ & $-0.0567$ & $+$ & $-$ & $1.095$ & $1.481$ \cr
7 & I & $0.0358$ & $-0.0617$ & $+$ & $-$ & $1.095$ & $1.549$ \cr
8 & I & $0.0404$ & $-0.0655$ & $+$ & $-$ & $1.095$ & $1.603$ \cr
10 & I & $0.0470$ & $-0.0710$ & $+$ & $-$ & $1.095$ & $1.682$ \cr
20 & I & $0.0611$ & $-0.0828$ & $+$ & $-$ & $1.095$ & $1.859$ \cr
50 & I & $0.0702$ & $-0.0903$ & $+$ & $-$ & $1.095$ & $1.971$
\end{tabular}
\end{center}
\label{tab:example2}
\end{table}%
%

\begin{remark} \label{re:G1_extended}
Consider a general 2d-QBD process. The results of Theorem \ref{th:mainresults} also hold for the case where several columns of $G_1(z)$ or $G_2(w)$ are zero. Here we only consider the case where several columns of $G_2(w)$ are zero. 
For simplicity, assume that, for some $k_0\in\{1,2,...,s_0-1\}$ and for any $w\in[\underline{z}_2^*,\bar{z}_2^*]$, the first $k_0$ columns of $G_2(w)$ are always zero. Further assume that the value of $0$ is the eigenvalue of $G_2(r_2)$ with algebraic multiplicity $k_0$ and the other $s_0-k_0$ eigenvalues of $G_2(r_2)$ are nonzero and distinct. 
$(1,K)$-limited service models in which the Poisson arrival processes are replaced with Markovian arrival processes are like this. 

For $w\in\mathbb{C}[\underline{z}_2^*,\bar{z}_2^*]$, let $\alpha_1(w)$, $\alpha_2(w)$, ..., $\alpha_{s_0}(w)$ be the solutions to $\phi(z,w)=0$, counting multiplicity, that satisfy $|\alpha_k(w)|\le \underline{\zeta}_1(w)$, $k=1,2,...,s_0$. 
By Lemma 4.3 of Ref.\ \cite{Ozawa18a}, these solutions are the eigenvalues of $G_2(w)$. Without loss of generality, we assume that $\alpha_k(w)\equiv 0$ for $k\in\{1,2,...,k_0\}$ and $\alpha_{s_0}(w)$ corresponds to $\underline{\zeta}_1(w)$ when $w\in[\underline{z}_2^*,\bar{z}_2^*]$. For $k\in\{k_0+1,k_0+2,...,s_0\}$, let $\bv_k(w)$ be the column vector satisfying $L(\alpha_k(w),w) \bv_k(w) = \bzero$. Each $\bv_k(w)$ is the right eigenvector of $G_2(w)$ with respect to the eigenvalue $\alpha_k(w)$. 
Define $s_0\times s_0$ matrix functions $J_2(w)$ and $V_2(w)$ as
\begin{align}
&J_2(w) = \diag(0,\dots,0,\alpha_{k_0+1}(w),\alpha_{k_0+2}(w),...,\alpha_{s_0}(w)),\\
&V_2(w) = \begin{pmatrix}  \be_1 & \be_2 & \cdots & \be_{k_0} & \bv_{k_0+1}(w) & \bv_{k_0+2}(w) & \cdots & \bv_{s_0}(w)\end{pmatrix}, 
\end{align}
If $\alpha_{k_0+1}(w), \alpha_{k_0+2}(w), ..., \alpha_{s_0}(w)$ are distinct, $V_2(w)$ is nonsingular and $G_2(w)$ is given as
\begin{equation}
G_2(w) = V_2(w) J_2(w) V_2(w)^{-1}. 
\end{equation}
Therefore, the results of Lemma 4.7 of Ref.\ \cite{Ozawa18a} 
hold for $G_2(w)$ and those in Section 5 of Ref.\ \cite{Ozawa18a} 
hold for $\bvarphi_2(w)$. As a result, the exact asymptotic formula $h_2(k)$ is given by Theorem \ref{th:mainresults}. 
\end{remark}

%
%
\section{Additional proofs for Ref.\ \cite{Ozawa18a}} \label{sec:proofs}

In this section, we give proofs of some lemmas and propositions appeared in Ref.\ \cite{Ozawa18a} as well as derivation of some coefficient vectors also appeared in Section 5 of Ref.\ \cite{Ozawa18a}.

%
%
\subsection{Proof of Lemma 2.2 of Ref.\ \cite{Ozawa18a}} \label{sec:proof_barGamma_bounded}

\noindent\textbf{Lemma 2.2 of Ref.\ \cite{Ozawa18a}.} 
{\it 
Under Assumption \ref{as:Akl_irreducible}, $\bar{\Gamma}$ is bounded. 
}

\begin{proof}
For $n\ge 1$, $j\in S_0$ and $(s_1,s_2)\in\mathbb{R}^2$, $C(e^{s_1},e^{s_2})^n$ satisfies 
\begin{align}
&\big[ C(e^{s_1},e^{s_2})^n \big]_{jj} 
= \sum_{\bk_{(n)}\in\mathbb{H}^n} \sum_{\bl_{(n)}\in\mathbb{H}^n} \big[ A_{k_1,l_1} A_{k_2,l_2} \cdots A_{k_n,l_n} \big]_{jj}\,e^{s_1\sum_{p=1}^n k_p+s_2\sum_{p=1}^n l_p},
\label{eq:Cs1s2n}
\end{align}
where $\bk_{(n)}=(k_1,k_2,...,k_n)$ and $\bl_{(n)}=(l_1,l_2,...,l_n)$.
Consider the Markov chain $\{\tilde{\bY}_n\}=\{(\tilde{X}_{1,n},\tilde{X}_{2,n},\tilde{J}_n)\}$ governed by $\{A_{k,l}:k,l\in\mathbb{H}\}$ (see Subsection \ref{sec:tdQBD}) and assume that $\{\tilde{\bY}_n\}$ starts from the state $(0,0,j)$. 
Since $\{\tilde{\bY}_n\}$ is irreducible and aperiodic, there exists $n_0\ge 1$ such that, for every $j\in S_0$, $\mathbb{P}(\tilde{\bY}_{n_0}=(1,0,j)\,|\,\tilde{\bY}_0=(0,0,j))>0$, $\mathbb{P}(\tilde{\bY}_{n_0}=(0,1,j)\,|\,\tilde{\bY}_0=(0,0,j))>0$, $\mathbb{P}(\tilde{\bY}_{n_0}=(-1,0,j)\,|\,\tilde{\bY}_0=(0,0,j))>0$ and $\mathbb{P}(\tilde{\bY}_{n_0}=(0,-1,j)\,|\,\tilde{\bY}_0=(0,0,j))>0$. 
This implies that, for some $\bk_{1,(n_0)}^{(j)}$, $\bl_{1,(n_0)}^{(j)}$, $\bk_{2,(n_0)}^{(j)}$, $\bl_{2,(n_0)}^{(j)}$, $\bk_{3,(n_0)}^{(j)}$, $\bl_{3,(n_0)}^{(j)}$, $\bk_{4,(n_0)}^{(j)}$ and $\bl_{4,(n_0)}^{(j)}$ in $\mathbb{H}^{n_0}$, we have 
\begin{align*}
&b_{1,j}=\big[ A_{k_{1,1}^{(j)},l_{1,1}^{(j)}} A_{k_{1,2}^{(j)},l_{1,2}^{(j)}} \cdots A_{k_{1,n_0}^{(j)},l_{1,n_0}^{(j)}} \big]_{jj}>0,\quad \sum_{p=1}^{n_0} k_{1,p}^{(j)}=1,\quad \sum_{p=1}^{n_0} l_{1,p}^{(j)} =0,\\
&b_{2,j}=\big[ A_{k_{2,1}^{(j)},l_{2,1}^{(j)}} A_{k_{2,2}^{(j)},l_{2,2}^{(j)}} \cdots A_{k_{2,n_0}^{(j)},l_{2,n_0}^{(j)}} \big]_{jj}>0,\quad \sum_{p=1}^{n_0} k_{2,p}^{(j)}=0,\quad \sum_{p=1}^{n_0} l_{2,p}^{(j)} =1,\\
&b_{3,j}=\big[ A_{k_{3,1}^{(j)},l_{3,1}^{(j)}} A_{k_{3,2}^{(j)},l_{3,2}^{(j)}} \cdots A_{k_{3,n_0}^{(j)},l_{3,n_0}^{(j)}} \big]_{jj}>0,\quad \sum_{p=1}^{n_0} k_{3,p}^{(j)}=-1,\quad \sum_{p=1}^{n_0} l_{3,p}^{(j)} =0,\\
&b_{4,j}=\big[ A_{k_{4,1}^{(j)},l_{4,1}^{(j)}} A_{k_{4,2}^{(j)},l_{4,2}^{(j)}} \cdots A_{k_{4,n_0}^{(j)},l_{4,n_0}^{(j)}} \big]_{jj}>0,\quad \sum_{p=1}^{n_0} k_{4,p}^{(j)}=0,\quad \sum_{p=1}^{n_0} l_{4,p}^{(j)} =-1.
\end{align*} 
Hence, the sum of any row of $C(e^{s_1},e^{s_2})^{n_0}$ is greater than or equal to $b^*(e^{s_1}+e^{s_2}+e^{-s_1}+e^{-s_2})$, where
\[
b^* = \min_{j\in S_0} \min_{1\le i\le 4} b_{i,j}>0, 
\]
and we obtain 
\begin{align}
\chi(e^{s_1},e^{s_2})
=\spr(C(e^{s_1},e^{s_2})) 
= \spr(C(e^{s_1},e^{s_2})^{n_0})^{\frac{1}{n_0}}
\ge \left( b^*(e^{s_1}+e^{s_2}+e^{-s_1}+e^{-s_2}) \right)^{\frac{1}{n_0}}, 
\end{align}
where we use the fact that, for a nonnegative square matrix $A=(a_{i,j})$, $\spr(A)\ge \min_i \sum_j a_{i,j}$ (see, for example, Theorem 8.1.22 of Horn and Johnson \cite{Horn85}). This means that $\chi(e^{s_1},e^{s_2})$ is unbounded in any direction, and since $\chi(e^{s_1},e^{s_2})$ is convex in $(s_1,s_2)$, we see that $\bar{\Gamma}$ is a bounded set. 
\end{proof}

%
%
\subsection{Proof of Lemma 2.5 of Ref.\ \cite{Ozawa18a}} 

\noindent\textbf{Lemma 2.5 of Ref.\ \cite{Ozawa18a}.} 
{\it 
If $a_2^{\{1,2\}}<0$, then for some positive constant $c_1$, $\psi_1'(1)=c_1 a_1^{\{1\}}$. 
If $a_1^{\{1,2\}}<0$, then for some positive constant $c_2$, $\psi_2'(1)=c_2 a_1^{\{2\}}$. 
}

\begin{proof}
We prove only the first claim. The second claim can be proved analogously. 
Assume $a_2^{\{1,2\}}<0$. Then, the induced Markov chain $\calL^{\{1\}}$ defined in Subsection \ref{sec:stationary_cond} is positive recurrent and $G_1(1)$ is stochastic.
For $z_1\in[\underline{z}_1^*,\bar{z}_1^*]$, $C_1(z_1,G_1(z_1))$ is a nonnegative matrix and $\psi_1(z_1)$ is its Perron-Frobenius eigenvalue. Let $\bu^{C_1}(z_1)$ and $\bv^{C_1}(z_1)$ be the Perron-Frobenius left and right eigenvectors of $C_1(z_1,G_1(z_1))$, satisfying $\bu^{C_1}(z_1) \bv^{C_1}(z_1)=1$. Since $\calL^{\{1\}}$ is positive recurrent, we have $\bu^{C_1}(1)=c_1 \bpi_{*,0}^{(1)}$ and $\bv^{C_1}(1)=\bone$, where $c_1=(\bpi_{*,0}^{(1)} \bone)^{-1}$ is the normalizing constant. 
The derivative of $\psi_1(z_1)$, $\psi_1'(z_1)$, is given as 
\begin{align}
&\psi_1'(z_1) = \bu^{C_1}(z_1) \left(\frac{d}{d z} C_1(z_1,G_1(z_1)) \right) \bv^{C_1}(z_1), \\
&\frac{d}{d z_1} C_1(z_1,G_1(z_1)) = -z_1^{-2}A_{-1,0}^{(1)}+A_{1,0}^{(1)} + (-z_1^{-2} A_{-1,1}^{(1)}+A_{1,1}^{(1)}) G_1(z_1) + A_{*,1}^{(1)}(z_1) G_1'(z_1). 
\end{align}
Since $G_1(z_1)$ is a solution to equation (\ref{eq:equationG1}) and $G_1(1)$ is stochastic, the derivative of $G_1(z_1)$, $G_1'(z_1)$, satisfies 
\begin{equation}
(-A_{-1,*}+A_{1,*}) \bone + (A_{*,0}+A_{*,1} G_1(1)+A_{*,1}) G_1'(1) \bone = G_1'(1) \bone. 
\label{eq:G1p1}
\end{equation}
Define a matrix $N_1$ as 
\begin{equation}
N_1 = \left( I-(A_{*,0}+A_{*,1} G_1(1)) \right)^{-1}. 
\end{equation}
This $N_1$ is well-defied since $G_1(1)$ is the G-matrix of the transition probability matrix $A_*^{(1)}$ (see Proposition 3.5 of Ozawa \cite{Ozawa13}). Furthermore, $N_1$ satisfies $G_1(1) = N_1 A_{*,-1}$ and $R_*^{(1)} = A_{*,1} N_1$.
From equation (\ref{eq:G1p1}), we, therefore, obtain
\begin{equation}
G_1'(1) \bone = N_1 (-A_{-1,*}+A_{1,*}) \bone + N_1 A_{*,1} G_1'(1) \bone, 
\end{equation}
and this leads us to 
\begin{align}
& G_1'(1) \bone = N_1 \sum_{k=0}^\infty (R_*^{(1)})^k (-A_{-1,*}+A_{1,*}) \bone.
\end{align}
As a result, we obtain 
\begin{align}
\psi_1'(1) &= c_1 \bpi_{*,0}^{(1)} \left( -A_{-1,0}^{(1)}+A_{1,0}^{(1)} -A_{-1,1}^{(1)}+A_{1,1}^{(1)} \right) \bone + c_1 \bpi_{*,0}^{(1)} A_{*,1}^{(1)} G_1'(1) \bone 
= c_1 a_1^{\{1\}},
\end{align}
where we use the fact that $\bpi_{*,1}^{(1)} = \bpi_{*,0}^{(1)} A_{*,1}^{(1)} N_1$, and this completes the proof. 
\end{proof}

%
%
\subsection{Proof of Lemma 3.2 of Ref.\ \cite{Ozawa18a}}

\noindent\textbf{Lemma 3.2 of Ref.\ \cite{Ozawa18a}.} 
{\it 
For $n\in\mathbb{Z}_+$, let $\ba_n$ be a $1\times m$ complex vector and define a vector series $\bphi(w)$ as $\bphi(w)=\sum_{n=0}^\infty \ba_n w^n$. Furthermore, for an $m\times m$ complex matrix $X$, define $\bphi(X)$ as $\bphi(X)=\sum_{n=0}^\infty \ba_n X^n$. 
Assume that $\bphi(w)$ converges absolutely for all $w\in\mathbb{C}$ such that $|w|<r$ for some $r>0$. Then, if $\spr(X)<r$, $\bphi(X)$ also converges absolutely. 
}
%
%
\begin{proof}
We denote by $J_k(\lambda)$ the $k$-dimensional Jordan block of eigenvalue $\lambda$. 
Note that the $n$-th power of $J_k(\lambda)$ is given by
\begin{equation}
J_k(\lambda)^n = \begin{pmatrix}
{}_nC_0\lambda^n & {}_nC_1\lambda^{n-1} & \cdots & {}_nC_{k-1}\lambda^{n-k+1} \cr
&  \ddots & \ddots & & \cr
&  & {}_nC_0\lambda^n & {}_nC_1\lambda^{n-1} \cr
& & & {}_nC_0\lambda^n 
\end{pmatrix}. 
\end{equation}
Without loss of generality, the Jordan canonical form $J$ of $X$ is represented as
\begin{equation}
J 
= T^{-1} X T 
= \begin{pmatrix}
J_{k_1}(\lambda_1) & & & \cr
 & J_{k_2}(\lambda_2) & & \cr
 &  & \ddots & \cr
 & & & J_{k_l}(\lambda_l) 
\end{pmatrix},
\end{equation}
where $T$ is a nonsingular matrix, $\lambda_1,\lambda_2,...,\lambda_l$ are the eigenvalues of $X$ and $k_1,k_2,...,k_l$ are positive integers satisfying $k_1+\cdots+k_l=m$. 
We have 
\begin{align}
\Big( \sum_{n=0}^\infty |\ba_n X^n| \Big)^\top 
&\le \sum_{n=0}^\infty |T^{-1}|^\top |J|^n |T|^\top |\ba_n|^\top \cr
&= \sum_{n=0}^\infty \left(\bone^\top\otimes |T^{-1}|^\top \right) \left( \diag(|\ba_n|)\otimes |J|^n \right) {\rm vec}(|T|^\top) \cr
&= \left(\bone^\top\otimes (|T^{-1}|)^\top \right) \diag\Big( \sum_{n=0}^\infty [|\ba_n|]_j |J|^n,\,j=1,2,...,s_0 \Big) {\rm vec}(|T|^\top), 
\end{align}
where we use the identity $\vecM(ABC)=(C^\top\otimes A) \vecM(B)$ for matrices $A$, $B$ and $C$. Note that 
\begin{align}
&\sum_{n=0}^\infty [|\ba_n|]_j |J|^n = \diag\Big( \sum_{n=0}^\infty [|\ba_n|]_j J_{k_s}(|\lambda_s|)^n,\,s=1,2,...,l \Big),  
\end{align}
and we have, for $t,u\in\{1,2,...,k_s\}$ such that $t\le u$,  
\begin{align}
&\sum_{n=0}^\infty \bigl[ [|\ba_n|]_j J_{k_s}(|\lambda_s|)^n \bigr]_{t,u} \cr
=&\ \sum_{n=0}^{k_s-2} \bigl[ [|\ba_n|]_j J_{k_s}(|\lambda_s|)^n \bigr]_{t,u} + \sum_{n=k_s-1}^\infty \frac{n!}{(u-t)! (n-u+t)!} [|\ba_n|]_j\,|\lambda_s|^{n-u+t} \cr
\le&\ \sum_{n=0}^{k_s-2} \bigl[ [|\ba_n|]_j J_{k_s}(|\lambda_s|)^n \bigr]_{t,u} + \sum_{n=u-t}^\infty \frac{n!}{(n-u+t)!} [|\ba_n|]_j\,|\lambda_s|^{n-u+t} \cr
=& \sum_{n=0}^{k_s-2} \bigl[ [|\ba_n|]_j J_{k_s}(|\lambda_s|)^n \bigr]_{t,u} + \frac{d^{u-t}}{d\,w^{u-t}}\,[\bphi_{abs}(w)]_j \Big|_{w=|\lambda_s|}, 
\label{eq:anJlambda}
\end{align}
where $\bphi_{abs}(w)=\sum_{n=0}^\infty |\ba_n|\,w^n$. For $w\in\mathbb{C}$ such that $|w|<r$, since $\bphi(w)$ is absolutely convergent, $\bphi_{abs}(w)$ is also absolutely convergent and analytic. Hence, $\bphi_{abs}(w)$ is differentiable any number of times and we know that the second term on the last line of formula (\ref{eq:anJlambda}) is finite since $|\lambda_s|\le \spr(X)<r$; this completes the proof.
\end{proof}

%
%
\subsection{Proof of Lemma 4.1 of Ref.\ \cite{Ozawa18a}} \label{sec:Lemma4_1}

For $z\in\mathbb{C}$, we define $G_1(z)$ in a manner similar to that used for defining so-called {\it G-matrices} of QBD process (see Neuts \cite{Neuts94}). For $z\in\mathbb{C}$, we also define $R_1(z)$ and $N_1(z)$. 
For the purpose, we use the following sets of index sequences: for $n\ge 1$ and for $m\ge 1$, 
\begin{align*}
&\scrI_n = \biggl\{\bi_{(n)}\in\mathbb{H}^n:\ \sum_{l=1}^k i_l\ge 0\ \mbox{for $k\in\{1,2,...,n-1\}$}\ \mbox{and} \sum_{l=1}^n i_l=0 \biggr\}, \\
&\scrI_{D,m,n} = \biggl\{\bi_{(n)}\in\mathbb{H}^n:\ \sum_{l=1}^k i_l\ge -m+1\ \mbox{for $k\in\{1,2,...,n-1\}$}\ \mbox{and} \sum_{l=1}^n i_l=-m \biggr\}, \\
&\scrI_{U,m,n} = \biggl\{\bi_{(n)}\in\mathbb{H}^n:\ \sum_{l=1}^k i_l\ge 1\ \mbox{for $k\in\{1,2,...,n-1\}$}\ \mbox{and} \sum_{l=1}^n i_l=m \biggr\}, 
\end{align*}
where $\bi_{(n)}=(i_1,i_2,...,i_n)$. 
For $n\ge 1$, define $Q_{11}^{(n)}(z)$, $D_1^{(n)}(z)$ and $U_1^{(n)}(z)$ as 
\begin{align*}
&Q_{11}^{(n)}(z) = \sum_{\bi_{(n)}\in\scrI_n} A_{*,i_1}(z) A_{*,i_2}(z) \cdots A_{*,i_n}(z), \\
&D_1^{(n)}(z) = \sum_{\bi_{(n)}\in\scrI_{D,1,n}} A_{*,i_1}(z) A_{*,i_2}(z) \cdots A_{*,i_n}(z), \\
&U_1^{(n)}(z) = \sum_{\bi_{(n)}\in\scrI_{U,1,n}} A_{*,i_1}(z) A_{*,i_2}(z) \cdots A_{*,i_n}(z), 
\end{align*}
and define $N_1(z)$, $R_1(z)$ and $G_1(z)$ as 
\begin{align*}
&N_1(z) = \sum_{n=0}^\infty Q_{11}^{(n)}(z),\quad 
G_1(z) = \sum_{n=1}^\infty D_1^{(n)}(z), \quad 
R_1(z) = \sum_{n=1}^\infty U_1^{(n)}(z),
\end{align*}
where $Q_{11}^{(0)}(z)=I$. 
$N_1(z)$, $G_1(z)$ and $R_1(z)$ satisfy the following properties. 

\bigskip\noindent\textbf{Lemma 4.1 of Ref.\ \cite{Ozawa18a}.} 
{\it 
For $z\in\mathbb{C}[\underline{z}_1^*,\bar{z}_1^*]$, the following statements hold.
\begin{itemize}
\item[(i)] $N_1(z)$, $G_1(z)$ and $R_1(z)$ converge absolutely and satisfy 
\begin{equation}
|N(z)|\le N(|z|),\quad 
|G_1(z)|\le G_1(|z|), \quad 
|R_1(z)|\le R_1(|z|). 
\label{eq:absNRG}
\end{equation}
\item[(ii)] $G_1(z)$ and $R_1(z)$ are represented in terms of $N_1(z)$, as follows. 
\begin{equation}
G_1(z) = N_1(z) A_{*,-1}(z),\quad 
R_1(z) = A_{*,1}(z) N_1(z).
\end{equation}
\item[(iii)] $G_1(z)$ and $R_1(z)$ satisfy the following matrix quadratic equations: 
\begin{align}
A_{*,-1}(z)+A_{*,0}(z) G_1(z)+A_{*,1}(z) G_1(z)^2 = G_1(z), \label{eq:G1equation} \\
R_1(z)^2 A_{*,-1}(z)+R_1(z) A_{*,0}(z)+A_{*,1}(z) = R_1(z). \label{eq:R1equation}
\end{align}
\item[(iv)] Define $H_1(z)$ as $H_1(z) = A_{*,0}(z)+A_{*,1}(z) N_1(z) A_{*,-1}(z)$, then $N_1(z)$ satisfies 
\begin{equation}
(I-H_1(z)) N_1(z) = I. 
\label{eq:H1N1_relation}
\end{equation}
\item[(v)] For nonzero $w\in\mathbb{C}$, $I-C(z,w)$ satisfies the following factorization (see, for example,  Ref.\ \cite{Miyazawa04}). 
\begin{equation}
I-C(z,w) = \bigl(w^{-1}I-R_1(z)\bigr) \bigl(I-H_1(z)\bigr) \bigl(w I-G_1(z)\bigr). 
\label{eq:WFfact}
\end{equation}
\end{itemize}
}
%
%
\begin{proof}
{\it (i)}\quad 
First, we note that if $z$ is a positive real number in $[\underline{z}_1^*,\bar{z}_1^*]$, then $N_1(z)$, $R_1(z)$ and $G_1(z)$ are finite. 
Let $z$ be a complex number satisfying $|z|\in[\underline{z}_1^*,\bar{z}_1^*]$. Since $A_{i,j},\,i,j\in\mathbb{H}$, are nonnegative, we have for $j\in\mathbb{H}$, 
\[
|A_{*,j}(z)| = \biggl| \sum_{i\in\mathbb{H}} A_{i,j} z^{i} \biggr| 
\le \sum_{i\in\mathbb{H}} A_{i,j} |z|^i = A_{*,j}(|z|), 
\]
and for $n\ge 0$,  
\[
|Q_{11}^{(n)}(z)| = \biggl| \sum_{\bi_{(n)}\in\scrI_n} A_{*,i_1}(z) A_{*,i_2}(z) \cdots A_{*,i_n}(z) \biggr|
\le Q_{11}^{(n)}(|z|). 
\]
From this and the fact that $N_1(|z|)=\sum_{n=0}^\infty Q_{11}^{(n)}(|z|)$ is finite (convergent), we see that $N_1(z)=\sum_{n=0}^\infty Q_{11}^{(n)}(z)$ converges absolutely and obtain $|N_1(z)|\le N_1(|z|)$. 
Analogously, we see that both $G_1(z)$ and $R_1(z)$ also converge absolutely and satisfy expression (\ref{eq:absNRG}).

{\it (ii)}\quad 
Since, for $n\ge 1$, $\scrI_{D,1,n}$ and $\scrI_{U,1,n}$ satisfy 
\begin{align*}
\scrI_{D,1,n} &= \biggl\{\bi_{(n)}\in\mathbb{H}^n:\ \sum_{l=1}^k i_l\ge 0\ \mbox{for $k\in\{1,2,...,n-2\}$},\ \sum_{l=1}^{n-1} i_l=0\ \mbox{and}\ i_n=-1 \biggr\} \cr
&= \bigl\{(\bi_{(n-1)},-1):\ \bi_{(n-1)}\in\scrI_{n-1} \bigr\}, \\
\scrI_{U,1,n} &= \biggl\{\bi_{(n)}\in\mathbb{H}^n:\ i_1=1,\,\sum_{l=2}^k i_l\ge 0\ \mbox{for $k\in\{2,...,n-1\}$}\ \mbox{and} \sum_{l=2}^n i_l=0 \biggr\} \cr 
&= \bigl\{(1,\bi_{(n-1)}):\ \bi_{(n-1)}\in\scrI_{n-1} \bigr\},
\end{align*}
where $\bi_{(n)}=(i_1,i_2,...,i_n)$, we have 
\begin{align*}
&G_1(z) = \sum_{n=1}^\infty Q_{11}^{(n-1)}(z) A_{*,-1}(z)  = N_1(z) A_{*,-1}(z), \\
&R_1(z) = \sum_{n=1}^\infty A_{*,1}(z) Q_{11}^{(n-1)}(z) = A_{*,1}(z) N_1(z). 
\end{align*}. 

{\it (iii)}\quad 
We prove only equation (\ref{eq:G1equation}) since equation (\ref{eq:R1equation}) can analogously be proved. 
For $n\ge 3$, $\scrI_{D,1,n}$ satisfies 
\begin{align*}
\scrI_{D,1,n} 
&= \biggl\{\bi_{(n)}\in\mathbb{H}^n:\ i_1=0,\ \sum_{l=2}^k i_l\ge 0\ \mbox{for $k\in\{2,...,n-1\}$},\ \sum_{l=2}^n i_l=-1 \biggr\} \cr
&\qquad \bigcup\,\biggl\{\bi_{(n)}\in\mathbb{H}^n:\ i_1=1,\ \sum_{l=2}^k i_l\ge -1\ \mbox{for $k\in\{2,...,n-1\}$},\ \sum_{l=2}^n i_l=-2 \biggr\} \cr
&= \bigl\{(0,\bi_{(n-1)}):\ \bi_{(n-1)}\in\scrI_{D,1,n-1} \bigr\} \cup \bigl\{(1,\bi_{(n-1)}):\ \bi_{(n-1)}\in\scrI_{D,2,n-1} \bigr\}, 
\end{align*}
and $\scrI_{D,2,n}$ satisfies 
\begin{align*}
\scrI_{D,2,n} 
&= \bigcup_{m=1}^{n-1} \biggl\{\bi_{(n)}\in\mathbb{H}^n:\ \sum_{l=1}^k i_l\ge 0\ \mbox{for $k\in\{1,2,...,m-1\}$},\ \sum_{l=1}^{m} i_l=-1,\cr
&\qquad\qquad\qquad \sum_{l=m+1}^k i_l\ge 0\ \mbox{for $k\in\{m+1,m+2,...,n-1\}$}\ \mbox{and} \sum_{l=m+1}^{n} i_l=-1\biggr\} \cr
&= \bigcup_{m=1}^{n-1} \bigl\{(\bi_{(m)},\bi_{(n-m)}):\ \bi_{(m)}\in\scrI_{D,1,m}\ \mbox{and}\ \bi_{(n-m)}\in\scrI_{D,1,m-n} \bigr\}. 
\end{align*}
Hence, we have, for $n\ge 3$, 
\begin{align*}
D_1^{(n)}(z) 
&= A_{*,0}(z) D_1^{(n-1)}(z) + A_{*,1}(z) \sum_{\bi_{(n-1)}\in\scrI_{D,2,n-1}} A_{*,i_1}(z) A_{*,i_2}(z) \cdots A_{*,i_{n-1}}(z) \cr
&= A_{*,0}(z) D_1^{(n-1)}(z) + A_{*,1}(z) \sum_{m=1}^{n-2} D_1^{(m)}(z) D_1^{(n-m-1)}(z), 
\end{align*} 
and obtain
\begin{align*}
G_1(z) 
&= D_1^{(1)}(z) + \sum_{n=2}^\infty A_{*,0}(z) D_1^{(n-1)}(z) + A_{*,1}(z) \sum_{n=3}^\infty \sum_{m=1}^{n-2} D_1^{(m)}(z) D_1^{(n-m-1)}(z) \cr
&= A_{*,-1}(z) + A_{*,0}(z) G_1(z) + A_{*,1}(z) G_1(z)^2, 
\end{align*}
where we use the fact that $D_1^{(1)}(z)=A_{*,-1}(z)$ and $D_1^{(2)}(z)=A_{*,0}(z) D_1^{(1)}(z)=A_{*,0}(z) A_{*,-1}(z)$. 

{\it (iv)}\quad 
For $n\ge 1$, $\scrI_{n}$ satisfies 
\begin{align*}
\scrI_{n} 
&= \biggl\{\bi_{(n)}\in\mathbb{H}^n:\ i_1=0,\ \sum_{l=2}^k i_l\ge 0\ \mbox{for $k\in\{2,...,n-1\}$},\ \sum_{l=2}^n i_l=0 \biggr\} \cr
&\qquad \bigcup\,\biggl\{\bi_{(n)}\in\mathbb{H}^n:\ i_1=1,\ \sum_{l=2}^k i_l\ge -1\ \mbox{for $k\in\{2,...,n-1\}$},\ \sum_{l=2}^n i_l=-1 \biggr\} \cr
&= \{(0,\bi_{(n-1)}):\ \bi_{(n-1)}\in\scrI_{n-1} \} \cr
&\qquad \bigcup \biggl( \bigcup_{m=2}^n\,\biggl\{\bi_{(n)}\in\mathbb{H}^n:\ i_1=1,\ \sum_{l=2}^k i_l\ge 0\ \mbox{for $k\in\{2,...,m-1\}$},\ \sum_{l=2}^m i_l=-1, \cr
&\qquad\qquad\qquad \sum_{l=m+1}^k i_l\ge 0\ \mbox{for $k\in\{m+1,...,n-1\}$},\ \sum_{l=m+1}^n i_l=0 \biggr\} \biggr) \cr
&= \{(0,\bi_{(n-1)}):\ \bi_{(n-1)}\in\scrI_{n-1} \} \cr
&\qquad \cup \left(\cup_{m=2}^n \{(1,\bi_{(m-1)},\bi_{(n-m)}):\ \bi_{(m-1)}\in\scrI_{D,1,m-1}, \bi_{(n-m)}\in\scrI_{n-m} \} \right),  
\end{align*}
where $\scrI_0=\emptyset$. Hence, we have, for $n\ge 1$, 
\[
Q_{1,1}^{(n)}(z) = A_{*,0}(z) Q_{1,1}^{(n-1)} + \sum_{m=2}^n A_{*,1}(z) D_1^{(m-1)}(z) Q_{1,1}^{(n-m)}(z), 
\] 
and this leads us to 
\begin{align*}
N_1(z) = I + A_{*,0}(z) N_1(z) + A_{*,1}(z) G_1(z) N_1(z).
\end{align*}
From this equation, we immediately obtain equation (\ref{eq:H1N1_relation}).

{\it (v)}\quad 
Substituting $N_1(z) A_{*,-1}(z)$, $A_{*,1}(z) N_1(z)$ and $A_{*,0}(z)+A_{*,1}(z) N_1(z) A_{*,-1}(z)$ for $G_1(z)$, $R_1(z)$ and $H_1(z)$, respectively, in the right hand side of equation (\ref{eq:WFfact}), we obtain the left hand side of the equation via straightforward calculation.
\end{proof}

%
%
\subsection{Proof of Proposition 4.1 of Ref.\ \cite{Ozawa18a}}

\noindent\textbf{Proposition 4.1 of Ref.\ \cite{Ozawa18a}.} 
{\it 
Under Assumption \ref{as:Akl_irreducible}, for $z,w\in\mathbb{C}$ such that $z\ne 0$ and $w\ne 0$, if $|z|\ne z$ or $|w|\ne w$, then $\spr(C(z,w))<\spr(C(|z|,|w|)$. 
}
%
\begin{proof}
Set $z=r e^{i \theta}$ and $w=r' e^{i \theta'}$, where $r,r'>0$, $\theta, \theta'\in[0,2\pi)$ and $i=\sqrt{-1}$. 
For $n\ge 1$ and $j\in S_0$, $C(z,w)^n$ satisfies 
\begin{align}
\Bigl| \big[ C(z,w)^n \big]_{jj} \Bigr|
&= \biggl| \sum_{\bk_{(n)}\in\mathbb{H}^n} \sum_{\bl_{(n)}\in\mathbb{H}^n} \big[ A_{k_1,l_1} A_{k_2,l_2} \cdots A_{k_n,l_n} \big]_{jj} \cr  
&\qquad\qquad\qquad\qquad \cdot r^{\sum_{p=1}^n k_p} (r')^{\sum_{p=1}^n l_p} e^{i (\theta \sum_{p=1}^n k_p + \theta' \sum_{p=1}^n l_p)} \biggr| \cr 
&\le \sum_{\bk_{(n)}\in\mathbb{H}^n} \sum_{\bl_{(n)}\in\mathbb{H}^n} \big[ A_{k_1,l_1} A_{k_2,l_2} \cdots A_{k_n,l_n} \big]_{jj}\,r^{\sum_{p=1}^n k_p} (r')^{\sum_{p=1}^n l_p} \cr
&= \big[ C(|z|,|w|)^n \big]_{jj}, 
\label{eq:Czwn}
\end{align}
where $\bk_{(n)}=(k_1,k_2,...,k_n)$ and $\bl_{(n)}=(l_1,l_2,...,l_n)$. In this formula, equality holds only when, for every $\bk_{(n)},\bl_{(n)}\in\mathbb{H}^n$ such that $\big[ A_{k_1,l_1} A_{k_2,l_2} \cdots A_{k_n,l_n} \big]_{jj}\ne 0$, $e^{i (\theta \sum_{p=1}^n k_p + \theta' \sum_{p=1}^n l_p)}$ takes some common value. 
Consider the Markov chain $\{\tilde{\bY}_n\}=\{(\tilde{X}_{1,n},\tilde{X}_{2,n},\tilde{J}_n)\}$ generated by $\{A_{k,l}:k,l\in\mathbb{H}\}$ (see Subsection \ref{sec:tdQBD}) and assume that $\{\tilde{\bY}_n\}$ starts from the state $(0,0,j)$. 
Since $\{\tilde{\bY}_n\}$ is irreducible and aperiodic, there exists $n_0\ge 1$ such that $\mathbb{P}(\tilde{\bY}_{n_0}=(0,0,j)\,|\,\tilde{\bY}_0=(0,0,j))>0$, $\mathbb{P}(\tilde{\bY}_{n_0}=(1,0,j)\,|\,\tilde{\bY}_0=(0,0,j))>0$ and $\mathbb{P}(\tilde{\bY}_{n_0}=(0,1,j)\,|\,\tilde{\bY}_0=(0,0,j))>0$. 
This implies that, for some $\bk_{(n_0)}$, $\bl_{(n_0)}$, $\bk_{(n_0)}'$, $\bl_{(n_0)}'$, $\bk_{(n_0)}''$ and $\bl_{(n_0)}''$ in $\mathbb{H}^{n_0}$, we have 
\begin{align*}
&\big[ A_{k_1,l_1} A_{k_2,l_2} \cdots A_{k_{n_0},l_{n_0}} \big]_{jj}>0,\quad \sum_{p=1}^{n_0} k_p=0,\quad \sum_{p=1}^{n_0} l_p =0,\\
&\big[ A_{k_1',l_1'} A_{k_2',l_2'} \cdots A_{k_{n_0}',l_{n_0}'} \big]_{jj}>0,\quad \sum_{p=1}^{n_0} k_p'=1,\quad \sum_{p=1}^{n_0} l_p' =0,\\
&\big[ A_{k_1'',l_1''} A_{k_2'',l_2''} \cdots A_{k_{n_0}'',l_{n_0}''} \big]_{jj}>0,\quad \sum_{p=1}^{n_0} k_p''=0,\quad \sum_{p=1}^{n_0} l_p'' =1. 
\end{align*} 
Hence, we see that if equality holds in formula (\ref{eq:Czwn}), then $e^{i \theta}=e^{i \theta'}=1$ and both $\theta$ and $\theta'$ must be zero. Therefore, if $\theta\ne 0$ or $\theta'\ne 0$, then we have 
\[
\big| [C(z,w)^{n_0}]_{jj} \big| = \Bigl| \big[ C(r e^{i\theta},r' e^{i\theta'})^{n_0} \big]_{jj} \Bigr| < \big[ C(r,r')^{n_0} \big]_{jj} = \big[ C(|z|,|w|)^{n_0} \big]_{jj}, 
\]
and obtain 
\begin{align}
\spr(C(z,w))^{n_0} &\le \spr(|C(z,w)^{n_0}|) < \spr(C(|z|,|w|)^{n_0}) = \spr(C(|z|,|w|))^{n_0},  
\end{align}
where we use Theorem 1.5 of Seneta \cite{Seneta06} and the fact that $C(|z|,|w|))$ is irreducible. Obviously, this implies that $\spr(C(z,w))<\spr(C(|z|,|w|))$. 
\end{proof}

%
%
\subsection{Proof of Propositions 5.1 and 5.2 of Ref.\ \cite{Ozawa18a}}

For $w\in\mathbb{C}$, $\bvarphi_2(w)$ is defined as $\bvarphi_2(w)=\sum_{j=1}^\infty \bnu_{0,j} w^j$. By Proposition 3.1 of Ref.\ \cite{Ozawa18a}, the radius of convergence of $\bvarphi_2(w)$ is given by $r_2$. 
For a scalar $z$ and matrix $X$, define vector functions $\bvarphi_2(X)$ and $\bvarphi_2^{\hat{C}_2}(z,X)$ as 
\begin{align*}
&\bvarphi_2(X) = \sum_{j=1}^\infty \bnu_{0,j} X^j, \quad 
\bvarphi_2^{\hat{C}_2}(z,X) = \sum_{j=1}^\infty \bnu_{0,j} \hat{C}_2(z,X) X^{j-1}, 
\end{align*}
where 
\[
\hat{C}_2(z,X) = \sum_{j\in\mathbb{H}} A_{*,j}^{(2)}(z) X^{j+1},\quad 
A_{*,j}^{(2)}(z) = \sum_{i\in\mathbb{H}_+} A_{i,j}^{(2)} z^i.
\]

\bigskip\noindent\textbf{Proposition 5.1 of Ref.\ \cite{Ozawa18a}.} 
{\it 
If $G_1(z)$ is entry-wise analytic and satisfies $\spr(G_1(z))<r_2$ in a region $\Omega\subset\mathbb{C}$, then $\bvarphi_2(G_1(z))$ and $\bvarphi_2^{\hat{C}_2}(z,G_1(z))$ are also entry-wise analytic in the same region $\Omega$. 
}
%
%
\begin{proof}
Let $X=(x_{k,l})$ be an $s_0\times s_0$ complex matrix. By Lemma 3.2 of Ref.\ \cite{Ozawa18a}, 
if $\spr(X)<r_2$, then $\bvarphi_2(X)$ converges absolutely. This means that each element of $\bvarphi_2(X)$ is an absolutely convergent series with $s_0^2$ valuables $x_{11},x_{12},...,x_{s_0,s_0}$ and it is analytic in the region $\{X=(x_{k,l})\in\mathbb{C}^{s_0^2}: \spr(X)<r_2\}$ as a function of $s_0^2$ complex valuables.
Therefore, if $G_1(z)$ is entry-wise analytic in $\Omega$ and $\spr(G_1(z))<r_2$, we see that $\bvarphi_2(G_1(z))=(\bvarphi_2\circ G_1)(z)$ is also entry-wise analytic in $\Omega$. 
For $z\in\mathbb{C}$ such that $|z|<r_2$, we have 
\[
\sum_{j=1}^\infty \Bigl| \bnu_{0,j} \hat{C}_2(z,G_1(z))\, z^{j-1} \Bigr|
\le |z|^{-1} \sum_{j=1}^\infty \bnu_{0,j} |z|^j\, \hat{C}_2(|z|,G_1(|z|))  
< \infty. 
\]
Hence, by Lemma 3.2 of Ref.\ \cite{Ozawa18a}, 
if $\spr(X)<r_2$, $\sum_{j=1}^\infty \bnu_{0,j} \hat{C}_2(z,G_1(z))\, X^{j-1}$ converges absolutely, and we see that if $G_1(z)$ is entry-wise analytic in $\Omega$ and $\spr(G_1(z))<r_2$, then $\bvarphi_2^{\hat{C}_2}(z,G_1(z))$ is entry-wise analytic in $\Omega$. 
\end{proof}

%
%

\bigskip
\noindent\textbf{Proposition 5.2 of Ref.\ \cite{Ozawa18a}.} 
{\it 
Under Assumption \ref{as:Akl_irreducible}, for any $z\in\mathbb{C}(\underline{z}_1^*,\bar{z}_1^*]$ such that $z\ne |z|$, we have 
\begin{equation}
\spr(C_1(z,G_1(z)))<\spr(C_1(|z|,G_1(|z|))). 
\end{equation}
}
\begin{proof}
Set $z=r e^{i\theta}$, where $r>0$, $\theta\in[0,2\pi)$ and $i=\sqrt{-1}$. By the definition of $G_1(z)$, we have
\begin{align}
C_1(z,G_1(z)) &= \sum_{n=1}^\infty\,\sum_{\bk_{(n+1)}\in\mathbb{H}^{n+1}}\,\sum_{\bl_{(n)}\in\scrI_{D,1,n}} A^{(1)}_{k_1,1} A_{k_2,l_1} A_{k_3,l_2} \cdots A_{k_{n+1},l_{n}} \cr
&\qquad\qquad\qquad\qquad\qquad \cdot r^{\sum_{p=1}^{n+1} k_p} e^{i (\theta \sum_{p=1}^{n+1} k_p)} + \sum_{k\in\mathbb{H}} A^{(1)}_{k,0} r^{k} e^{i\theta k}, 
\label{eq:C1zG1_Y1n}
\end{align}
where $\bk_{(n+1)}=(k_1,k_2,...,k_{n+1})$ and $\bl_{(n)}=(l_1,l_2,...,l_n)$.
Consider the Markov chain $\{\tilde{\bY}^{(1)}_n\}=\{(\tilde{X}^{(1)}_{1,n},\tilde{X}^{(1)}_{2,n},\tilde{J}^{(1)}_n)\}$ generated by $\{ \{A_{k,l}:k,l\in\mathbb{H}\}, \{A^{(1)}_{k,l}:k\in\mathbb{H}, l\in\mathbb{H}_+\} \}$ (see Subsection \ref{sec:tdQBD}), and assume that $\{\tilde{\bY}^{(1)}_n\}$ starts from a state in $\{0\}\times\{0\}\times S_0$. 
For $n_0\in\mathbb{N}$, let $\tau_{n_0}$ be the time when the Markov chain enters a state in $\mathbb{Z}\times\{0\}\times S_0$ for the $n_0$-th time. Then, the term $\sum_{p=1}^{n+1} k_p$ (resp.\ $k$) in expression (\ref{eq:C1zG1_Y1n}) indicates that $\tilde{X}^{(1)}_{1,\tau_1}=\sum_{p=1}^{n+1} k_p$ (resp.\ $\tilde{X}^{(1)}_{1,\tau_1}=k$). 
This point analogously holds for $\tilde{X}^{(1)}_{1,\tau_{n_0}}$ when $n_0>1$. Hence, $C_1(z,G_1(z))^{n_0}$ can be represented as 
\begin{align}
C_1(z,G_1(z))^{n_0} = \sum_{k\in\mathbb{Z}} \tilde{D}_k r^k e^{i\theta k}, 
\label{eq:C1zG1_Dre1}
\end{align}
where $k$ indicates that $\tilde{X}^{(1)}_{1,\tau_{n_0}}=k$ and the $(j,j')$-entry of nonnegative square matrix $\tilde{D}_k$ is given as $[\tilde{D}_k]_{jj'}=\mathbb{P}(\tilde{\bY}^{(1)}_{\tau_{n_0}}=(k,0,j')\,|\,\tilde{\bY}^{(1)}_0=(0,0,j))$. 
We have 
\begin{align}
|C_1(z,G_1(z))^{n_0}| = \left|\sum_{k\in\mathbb{Z}} \tilde{D}_k r^k e^{i\theta k}\right| \le C_1(|z|,G_1(|z|))^{n_0}, 
\label{eq:C1zG1_Dre2}
\end{align}
and equality holds only when, for each $j,j'\in S_0$ and for every $k\in\mathbb{Z}$ such that $[\tilde{D}_k]_{jj'}>0$, $e^{i\theta k}$ takes some common value. 
Under Assumption \ref{as:Akl_irreducible}, $\{\tilde{\bY}^{(1)}_n\}$ is irreducible and aperiodic, and for any $j,j'\in S_0$, there exists $n_0\ge 1$ such that $\mathbb{P}(\tilde{\bY}^{(1)}_{\tau_{n_0}}=(0,0,j')\,|\,\tilde{\bY}^{(1)}_0=(0,0,j))>0$ and $\mathbb{P}(\tilde{\bY}^{(1)}_{\tau_{n_0}}=(1,0,j')\,|\,\tilde{\bY}^{(1)}_0=(0,0,j))>0$. 
This implies that $[\tilde{D}_0]_{jj'}>0$ and $[\tilde{D}_1]_{jj'}>0$. Hence, if $\theta\ne 0$ ($z\ne |z|$), then we have 
\[
|[C_1(z,G_1(z))^{n_0}]_{jj'}| < [C_1(|z|,G_1(|z|))^{n_0}]_{jj'}, 
\label{eq:C1zG1_Dre3}
\]
and obtain 
\begin{align}
\spr(C_1(z,G_1(z))^{n_0}) \le \spr(|C_1(z,G_1(z))^{n_0}|) < \spr(C_1(|z|,G_1(|z|))^{n_0}),  
\end{align}
where we use Theorem 1.5 of Seneta \cite{Seneta06} and the fact that $C_1(|z|,G_1(|z|))$ is irreducible. Obviously, this implies that $\spr(C_1(z,G_1(z))) < \spr(C_1(|z|,G_1(|z|)))$ when $z\ne |z|$. 
\end{proof}

%
%
\subsection{Proof of Propositions 5.6 and 5.7 of Ref.\ \cite{Ozawa18a}}

%
%

%
Let $\alpha_j(z),\,j=1,2,...,s_0,$ be the eigenvalues of $G_1(z)$, counting multiplicity, and without loss of generality, let $\alpha_{s_0}(z)$ be the eigenvalue corresponding to $\underline{\zeta}_2(z)$ when $z\in[\underline{z}_1^*,\bar{z}_1^*]$. For $j\in\{1,2,...,s_0\}$, let $\bv_j(z)$ be the right eigenvector of $G_1(z)$ with respect to the eigenvalue $\alpha_j(z)$. 
Define matrices $J_1(z)$ and $V_1(z)$ as $J_1(z) = \diag(\alpha_1(z),\alpha_2(z),...,\alpha_{s_0}(z))$ and $V_1(z)=\begin{pmatrix} \bv_1(z) & \bv_2(z) & \cdots & \bv_{s_0}(z) \end{pmatrix}$, respectively. Under Assumption \ref{as:G1_eigen_z1max}, $V_1(z)$ is nonsingular in a neighborhood of $z=r_1$, and $G_1(z)$ is factorized as $G_1(z)=V_1(z) J_1(z) V_1(z)^{-1}$ (Jordan decomposition).  
We represent $V_1(z)^{-1}$ as 
\[
V_1(z)^{-1} = \begin{pmatrix} \bu_1(z) \cr \bu_2(z) \cr \vdots \cr \bu_{s_0}(z) \end{pmatrix}, 
\]
where each $\bu_j(z)$ is a $1\times s_0$ vector. For $j\in\{1,2,...,s_0\}$, $\bu_j(z)$ is the left eigenvector of $G_1(z)$ with respect to the eigenvalue $\alpha_j(z)$ and satisfies $\bu_j(z) \bv_k(z)=\delta_{jk}$ for $k\in\{1,2,...,s_0\}$, where $\delta_{jk}$ is the Kronecker delta. 
Under Assumption \ref{as:G1_eigen_z1max}, if $r_1<\bar{z}_1^*$, $J_1(z)$ and $V_1(z)$ are entry-wise analytic in a neighborhood of $z=r_1$; if $r_1=\bar{z}_1^*$, $J_1(z)$ and $V_1(z)$ except for $\alpha_{s_0}(z)$ and $\bv_{s_0}(z)$ are also entry-wise analytic in a neighborhood of $z=\bar{z}_1^*$, where $\alpha_{s_0}(z)$ and $\bv_{s_0}(z)$ are given in terms of a function $\tilde{\alpha}_{s_0}(\zeta)$ and vector function $\tilde{\bv}_{s_0}(\zeta)$ analytic in a neighborhood of the origin as $\alpha_{s_0}(z)=\tilde{\alpha}_{s_0}((\bar{z}_1^*-z)^{\frac{1}{2}})$ and $\bv_{s_0}(z)=\tilde{\bv}_{s_0}((\bar{z}_1^*-z)^{\frac{1}{2}})$ (see Section 4 of Ref.\ \cite{Ozawa18a}). 
%

Define $\tilde{J}_1(\zeta)$ and $\tilde{V}_1(\zeta)$ as 
\begin{align*}
&\tilde{J}_1(\zeta) = \diag\big(\alpha_1(\bar{z}_1^*-\zeta^2),...,\alpha_{s_0-1}(\bar{z}_1^*-\zeta^2),\tilde{\alpha}_{s_0}(\zeta) \big), \\
&\tilde{V}_1(\zeta) = \begin{pmatrix} \bv_1(\bar{z}_1^*-\zeta^2) & \cdots & \bv_{s_0-1}(\bar{z}_1^*-\zeta^2) & \tilde{\bv}_{s_0}(\zeta) \end{pmatrix}.
\end{align*}
Under Assumption \ref{as:G1_eigen_z1max}, $\tilde{J}_1(\zeta)$ and $\tilde{V}_1(\zeta)$ are entry-wise analytic in a neighborhood of the origin and $\tilde{V}_1(\zeta)$ is nonsingular. Therefore, $\tilde{G}_1(\zeta)$ defined as $\tilde{G}_1(\zeta) = \tilde{V}_1(\zeta) \tilde{J}_1(\zeta) \tilde{V}_1(\zeta)^{-1}$ is entry-wise analytic in a neighborhood of the origin and satisfies 
\begin{equation}
G_1(z) = \tilde{G}_1((\bar{z}_1^*-z)^{\frac{1}{2}}). 
\label{eq:tildeG1}
\end{equation}
Since $\tilde{\alpha}_{s_0}(\zeta)$ and $\tilde{G}_1(\zeta)$ are analytic in a neighborhood of the origin, their Taylor series are represented as 
\begin{align}
\tilde{\alpha}_{s_0}(\zeta) = \sum_{k=0}^\infty \tilde{\alpha}_{s_0,k}\, \zeta^k, \quad 
\tilde{G}_1(\zeta) = \sum_{k=0}^\infty \tilde{G}_{1,k}\, \zeta^k, 
\end{align}
where $\tilde{\alpha}_{s_0,0}=\underline{\zeta}_2(\bar{z}_1^*)$ and $\tilde{G}_{1,0}=G_1(\bar{z}_1^*)$. The Puiseux series for $G_1(z)$ at $z=\bar{z}_1^*$ is represented as 
\[
G_1(z) = \sum_{k=0}^\infty \tilde{G}_{1,k}\, (\bar{z}_1^*-z)^{\frac{k}{2}}. 
\]

\medskip
In the following proposition and its proof,  $\tilde{\alpha}_{s_0,1}$, $\bu_{s_0}(z)$ and $\bv_{s_0}(z)$ are denoted by $\tilde{\alpha}_{s_0,1}^{G_1}$, $\bu_{s_0}^{G_1}(z)$ and $\bv_{s_0}^{G_1}(z)$, respectively, in order to explicitly indicate that they are a coefficient and vector with respect to $G_1(z)$. This notation is also used in Proposition 5.7 of Ref.\ \cite{Ozawa18a} and Proposition \ref{pr:limit_varphi2}. 

%
\bigskip\noindent\textbf{Proposition 5.6 of Ref.\ \cite{Ozawa18a}.} 
{\it 
\begin{align}
\lim_{\tilde{\Delta}_{\bar{z}_1^*}\ni z\to \bar{z}_1^*} \frac{G_1(\bar{z}_1^*)-G_1(z)}{(\bar{z}_1^*-z)^{\frac{1}{2}}} 
&= -\tilde{G}_{1,1} 
=  -\tilde{\alpha}_{s_0,1}^{G_1} N_1(\bar{z}_1^*) \bv^{R_1}(\bar{z}_1^*) \bu_{s_0}^{G_1}(\bar{z}_1^*) \ge \bzero^\top,\ \ne\bzero^\top, 
\label{eq:limit_G1}
\end{align}
where $\bv^{R_1}(\bar{z}_1^*)$ is the right eigenvector of $R_1(\bar{z}_1^*)$ with respect to the eigenvalue $\bar{\zeta}_2(\bar{z}_1^*)^{-1}$, satisfying $\bu_{s_0}^{G_1}(\bar{z}_1^*) N_1(\bar{z}_1^*) \bv^{R_1}(\bar{z}_1^*) =1$. 
}
%

\begin{proof}
$\tilde{G}_{1,1}$ is given by $\tilde{G}_{1,1}=(d/d\zeta)\,\tilde{G}_1(\zeta)\,|_{\zeta=0}$.
Differentiating both sides of $\tilde{G}_1(\zeta)=\tilde{V}_1(\zeta)\tilde{J}_1(\zeta)\tilde{V}_1(\zeta)^{-1}$ and setting $\zeta=0$, we obtain 
\begin{equation}
\tilde{G}_{1,1} = \tilde{\alpha}_{s_0,1}^{G_1}\, \bv^\dagger\, \bu_{s_0}^{G_1}(\bar{z}_1^*), 
\label{eq:tildeG1_diff1}
\end{equation}
where 
$\bv^\dagger = \bv_{s_0}^{G_1}(\bar{z}_1^*) + (\tilde{\alpha}_{s_0,1}^{G_1} )^{-1} ( \underline{\zeta}_2(\bar{z}_1^*) I - G_1(\bar{z}_1^*) ) \tilde{\bv}_{s_0,1}$ and 
$\tilde{\bv}_{s_0,1} = (d/d\zeta)\,\tilde{\bv}_{s_0}(\zeta)|_{\zeta=0}$. In the derivation of equation (\ref{eq:tildeG1_diff1}), we use the following identity:
\[
\frac{d}{d\,\zeta} \tilde{V}_1(\zeta)^{-1} = -\tilde{V}_1(\zeta)^{-1} \left(\frac{d}{d\,\zeta} \tilde{V}_1(\zeta)\right) \tilde{V}_1(\zeta)^{-1}. 
\]
Note that $\bu_{s_0}^{G_1}(\bar{z}_1^*) \bv^\dagger = 1$ and $\tilde{\alpha}_{s_0,1}^{G_1}$ is an eigenvalue of $\tilde{G}_{1,1}$. 
By equation (\ref{eq:G1equation}), $\tilde{G}_1(\zeta)$ satisfies 
\begin{equation}
\tilde{G}_1(\zeta) = A_{*,-1}(\bar{z}_1^*-\zeta^2)+A_{*,0}(\bar{z}_1^*-\zeta^2) \tilde{G}_1(\zeta)+A_{*,1}(\bar{z}_1^*-\zeta^2) \tilde{G}_1(\zeta)^2. 
\label{eq:tildeG1_equation}
\end{equation}
Differentiating both sides of equation (\ref{eq:tildeG1_equation}) and setting $\zeta=0$, we obtain 
\begin{equation}
\tilde{G}_{1,1} = A^\dagger\, \tilde{G}_{1,1}, 
\label{eq:tildeG1_diff2}
\end{equation}
where $A^\dagger=A_{*,0}(\bar{z}_1^*) + \underline{\zeta}_2(\bar{z}_1^*) A_{*,1}(\bar{z}_1^*) + A_{*,1}(\bar{z}_1^*) G_1(\bar{z}_1^*)$. In the derivation of equation (\ref{eq:tildeG1_diff2}), we use the fact that $\tilde{G}_{1,1} G_1(\bar{z}_1^*) = \underline{\zeta}_2(\bar{z}_1^*) \tilde{G}_{1,1}$.
Multiplying both sides of equation (\ref{eq:tildeG1_diff2}) by $\bv_{s_0}^{G_1}(\bar{z}_1^*)$ from the right, we obtain $A^\dagger \bv^\dagger = \bv^\dagger$. Hence, $\bv^\dagger$ is the right eigenvector of $A^\dagger$ with respect to the eigenvalue 1. 
%
From equation (\ref{eq:WFfact}), we obtain 
\begin{equation}
I-A^\dagger = (\bar{\zeta}_2(\bar{z}_1^*)^{-1}I-R_1(\bar{z}_1^*)) (I-H_1(\bar{z}_1^*)),
\end{equation}
and from equation (\ref{eq:H1N1_relation}), we know that $N_1(\bar{z}_1^*)=(I-H_1(\bar{z}_1^*))^{-1}$. Hence, $N_1(\bar{z}_1^*) \bv^{R_1}(\bar{z}_1^*)$ is the right eigenvector of $A^\dagger$ with respect to the eigenvalue 1, and we see that $\bv^\dagger$ can be given by $\bv^\dagger = N_1(\bar{z}_1^*) \bv^{R_1}(\bar{z}_1^*)$. 
Since $\bu_{s_0}^{G_1}(\bar{z}_1^*)$ is the left eigenvector of $\tilde{G}_{1,1}$ with respect to the eigenvalue $\tilde{\alpha}_{s_0,1}^{G_1}$, $\bv^{R_1}(\bar{z}_1^*)$ must satisfy $\bu_{s_0}^{G_1}(\bar{z}_1^*) N_1(\bar{z}_1^*) \bv^{R_1}(\bar{z}_1^*)=1$. 
Since $\tilde{\alpha}_{s_0,1}^{G_1}$ is negative, nonnegativity of $-\tilde{G}_{1,1}$ is obvious. Since $N_1(\bar{z}_1^*)\ge I$, we have $N_1(\bar{z}_1^*) \bv^{R_1}(\bar{z}_1^*)\ge \bv^{R_1}(\bar{z}_1^*)\ne \bzero$ and $-\tilde{G}_{1,1}$ is nonzero. 
\end{proof}

%
%

\bigskip
From the results in Section 5 of Ref.\ \cite{Ozawa18a}, we know that $\bvarphi_2(G_1(z))$ is represented as 
\begin{align}
\bvarphi_2(G_1(z)) = \begin{pmatrix} \bvarphi_2(\alpha_1(z))\bv_1(z) & \bvarphi_2(\alpha_2(z))\bv_2(z) & \cdots & \bvarphi_2(\alpha_{s_0}(z))\bv_{s_0}(z) \end{pmatrix} V(z)^{-1}. 
\label{eq:varphi2G1_extension1}
\end{align} 
Assume $r_1<\bar{z}_1^*$ and $\spr(G_1(r_1))=r_2<\bar{z}_2^*$. Then, the Laurent series for $\bvarphi_2(G_1(z))$ at $z=r_1$ is represented as
\begin{equation}
\bvarphi_2(G_1(z)) = \sum_{k=-1}^\infty \bvarphi_{2,k}^{G_1}\, (r_1-z)^k. 
\label{eq:varphi2G1_1}
\end{equation}
Assume $r_1=\bar{z}_1^*$ and define $\tilde{\bvarphi}_2(\tilde{G}_1(\zeta))$ as 
\begin{align}
\tilde{\bvarphi}_2(\tilde{G}_1(\zeta)) 
&= \Bigl( \bvarphi_2(\alpha_1(\bar{z}_1^*-\zeta^2))\bv_1(\bar{z}_1^*-\zeta^2)\ \  \cdots \cr
&\qquad \bvarphi_2(\alpha_{s_0-1}(\bar{z}_1^*-\zeta^2))\bv_{s_0-1}(\bar{z}_1^*-\zeta^2)\ \ \bvarphi_2(\tilde{\alpha}_{s_0}(\zeta))\tilde{\bv}_{s_0}(\zeta) \Bigr) \tilde{V}(\zeta)^{-1}. 
\label{eq:varphi2G1_extension2}
\end{align} 
If $\spr(\tilde{G}_1(0))=\spr(G_1(\bar{z}_1^*))<r_2$, then the Puiseux series for $\bvarphi_2(G_1(z))$ at $z=\bar{z}_1^*$ is represented as
\begin{align}
\bvarphi_2(G_1(z)) 
= \tilde{\bvarphi}_2(\tilde{G}_1((\bar{z}_1^*-z)^{\frac{1}{2}})) 
= \sum_{k=0}^\infty \tilde{\bvarphi}_{2,k}^{\tilde{G}_1}\, (\bar{z}_1^*-z)^{\frac{k}{2}}, 
\label{eq:varphi2G1_2}
\end{align}
where $\tilde{\bvarphi}_{2,0}^{\tilde{G}_1}=\bvarphi_2(G_1(\bar{z}_1^*))$. If $\spr(\tilde{G}_1(0))=\spr(G_1(\bar{z}_1^*))=r_2$, then the Puiseux series for $\bvarphi_2(G_1(z))$ at $z=\bar{z}_1^*$ is represented as
\begin{align}
\bvarphi_2(G_1(z)) 
= \tilde{\bvarphi}_2(\tilde{G}_1((\bar{z}_1^*-z)^{\frac{1}{2}})) 
= \sum_{k=-1}^\infty \tilde{\bvarphi}_{2,k}^{\tilde{G}_1}\, (\bar{z}_1^*-z)^{\frac{k}{2}}. 
\label{eq:varphi2G1_3}
\end{align}


In the following proposition, $\bu^{C_2}(w)$ is the left eigenvector of $C_2(G_2(w),w)$ with respect to the eigenvalue $\psi_2(w)$. For the definition of $c_{pole}^{\bvarphi_2}$, see Corollary 5.1 of Ref.\ \cite{Ozawa18a}. 

%
\bigskip\noindent\textbf{Proposition 5.7 of Ref.\ \cite{Ozawa18a}.} 
{\it 
In the case of Type I, if $r_1=\bar{z}_1^*$, then $\spr(G_1(\bar{z}_1^*))<r_2$ and $\tilde{\bvarphi}_{2,1}^{\tilde{G}_1}$ in expression (\ref{eq:varphi2G1_2}) is given by 
\begin{align}
\tilde{\bvarphi}_{2,1}^{\tilde{G}_1} 
&= \sum_{k=1}^\infty \bnu_{0,k} \sum_{l=1}^k \underline{\zeta}_2(\bar{z}_1^*)^{k-l} G_1(\bar{z}_1^*)^{l-1} \tilde{G}_{1,1} \le \bzero^\top,\ \ne \bzero^\top. 
\label{eq:limit_varphi2G1_z1max_1}
\end{align}
In the case of Type II, if $r_1<\bar{z}_1^*$, then $\spr(G_1(r_1))=r_2<\bar{z}_2^*$ and $\bvarphi_{2,-1}^{G_1}$ in expression (\ref{eq:varphi2G1_1}) is given by 
\begin{align}
\bvarphi_{2,-1}^{G_1} 
&= \underline{\zeta}_{2,z}(r_1)^{-1}\,c_{pole}^{\bvarphi_2} \bu^{C_2}(r_2)\,\bv_{s_0}^{G_1}(r_1) \bu_{s_0}^{G_1}(r_1) \ge \bzero^\top,\ \ne \bzero^\top,  
\label{eq:limit_varphi2G1_r1_II_1}
\end{align}
where $\underline{\zeta}_{2,z}(z)=(d/d z) \underline{\zeta}_2(z)$; if $r_1=\bar{z}_1^*$, then $\spr(G_1(\bar{z}_1^*))=r_2<\bar{z}_2^*$ and $\tilde{\bvarphi}_{2,-1}^{\tilde{G}_1}$ in expression (\ref{eq:varphi2G1_3}) is given by 
\begin{align}
\tilde{\bvarphi}_{2,-1}^{\tilde{G}_1} 
&= (-\tilde{\alpha}_{s_0,1}^{G_1})^{-1}\,c_{pole}^{\bvarphi_2} \bu^{C_2}(r_2)\,\bv_{s_0}^{G_1}(\bar{z}_1^*) \bu_{s_0}^{G_1}(\bar{z}_1^*) \ge \bzero^\top,\ \ne \bzero^\top. 
\label{eq:limit_varphi2G1_z1max_II_2}
\end{align}
}

\bigskip
Before proving Proposition 5.7 of Ref.\ \cite{Ozawa18a}, we give the following proposition. 
%
\begin{proposition} \label{pr:limit_varphi2} 
In the case of Type II, if $r_1<\bar{z}_1^*$, then $\underline{\zeta}_2(r_1)=r_2<\bar{z}_2^*$ and 
\begin{align}
\lim_{z\to r_1} (r_1-z) \bvarphi_2(\alpha_{s_0}(z)) 
&= \underline{\zeta}_{2,z}(r_1)^{-1}\,c_{pole}^{\bvarphi_2}\, \bu^{C_2}(r_2) > \bzero^\top;  
\label{eq:limit_varphi2_r1}
\end{align}
if $r_1=\bar{z}_1^*$, then $\underline{\zeta}_2(\bar{z}_1^*)=r_2<\bar{z}_2^*$ and 
\begin{align}
\lim_{\tilde{\Delta}_{\bar{z}_1^*}\ni z\to \bar{z}_1^*} (\bar{z}_1^*-z)^{\frac{1}{2}} \bvarphi_2(\alpha_{s_0}(z)) 
&= (-\tilde{\alpha}_{s_0,1}^{G_1})^{-1}\,c_{pole}^{\bvarphi_2}\, \bu^{C_2}(r_2) > \bzero^\top. 
\label{eq:limit_varphi2_z1max}
\end{align}
\end{proposition}
\begin{proof}
In the case of Type II, we have $\alpha_{s_0}(r_1)=\underline{\zeta}_2(e^{\bar{\eta}_1^{(c)}})=e^{\eta_2^{(c)}}=r_2<\bar{z}_2^*$, and the point $w=r_2$ is a pole of $\bvarphi_2(w)$ with order one. 
If $r_1<\bar{z}_1^*$, $\alpha_{s_0}(z)$ is analytic at $z=r_1$ and, by Corollary 5.1 of Ref.\ \cite{Ozawa18a}, we have 
\begin{align}
\lim_{z\to r_1} (r_1-z)\,\bvarphi_2(\alpha_{s_0}(z)) 
&= \lim_{z\to r_1} (r_2-\alpha_{s_0}(z))\,\bvarphi_2(\alpha_{s_0}(z))\, \frac{r_1-z}{r_2-\alpha_{s_0}(z)} \cr
&= c_{pole}^{\bvarphi_2}\, \bu^{C_2}(r_2)/\underline{\zeta}_{2,z}(r_1), 
\end{align}
and if $r_1=\bar{z}_1^*$, 
\begin{align}
\lim_{\tilde{\Delta}_{\bar{z}_1^*}\ni z\to \bar{z}_1^*} (\bar{z}_1^*-z)^{\frac{1}{2}}\,\bvarphi_2(\alpha_{s_0}(z)) 
&= \lim_{\tilde{\Delta}_{\bar{z}_1^*}\ni z\to \bar{z}_1^*} (r_2-\alpha_{s_0}(z))\,\bvarphi_2(\alpha_{s_0}(z))\, \biggl( - \frac{(\bar{z}_1^*-z)^\frac{1}{2}}{\alpha_{s_0}(z)-r_2} \biggr) \cr
&= c_{pole}^{\bvarphi_2}\, \bu^{C_2}(r_2) /(-\tilde{\alpha}_{s_0,1}^{G_1}), 
\end{align}
where we use Proposition 5.5 of Ref.\ \cite{Ozawa18a}. 
\end{proof}

%
\medskip
\begin{proof}[Proof of Proposition 5.7 of Ref.\ \cite{Ozawa18a}]
In the case of Type I, we always have $\spr(G_1(r_1))=\underline{\zeta}_2(r_1)<r_2$. Assume $r_1=\bar{z}_1^*$. By the definition of $\tilde{\bvarphi}_2(\tilde{G}_1(\zeta))$, we have $\tilde{\bvarphi}_2(\tilde{G}_1(\zeta))=\bvarphi_2(\tilde{G}_1(\zeta))$. Therefore, by Proposition 5.1 of Ref.\ \cite{Ozawa18a}, since $\tilde{G}_1(0)=\spr(G_1(\bar{z}_1^*))<r_2$, $\tilde{\bvarphi}_2(\tilde{G}_1(\zeta))$ is given in a form of absolutely convergent series as
\begin{equation}
\tilde{\bvarphi}_2(\tilde{G}_1(\zeta)) = \sum_{k=0}^\infty \bnu_{0,k} \tilde{G}_1(\zeta)^k.
\end{equation}
This $\tilde{\bvarphi}_2(\tilde{G}_1(\zeta))$ is entry-wise analytic at $\zeta=0$ and we have
\begin{align}
\tilde{\bvarphi}_{2,1}^{\tilde{G}_1} 
= \frac{d}{d \zeta} \tilde{\bvarphi}_2(\tilde{G}_1(\zeta)) \Big|_{\zeta=0}
&= \sum_{k=1}^\infty \bnu_{0,k} \sum_{l=1}^k G_1(\bar{z}_1^*)^{l-1} \tilde{G}_{1,1} G_1(\bar{z}_1^*)^{k-l} \cr
&= \sum_{k=1}^\infty \bnu_{0,k} \sum_{l=1}^k \underline{\zeta}_2(\bar{z}_1^*)^{k-l} G_1(\bar{z}_1^*)^{l-1} \tilde{G}_{1,1},
\end{align}
where we use the fact that 
\[
\tilde{G}_{1,1} G_1(\bar{z}_1^*) = \tilde{\alpha}_{s_0,1}^{G_1} N_1(\bar{z}_1^*) \bv^{R_1}(\bar{z}_1^*) \bu_{s_0}^{G_1}(\bar{z}_1^*) G_1(\bar{z}_1^*) = \underline{\zeta}_2(\bar{z}_1^*) \tilde{G}_{1,1}. 
\]
Since $\tilde{G}_{1,1}$ is nonzero and nonpositive, $\tilde{\bvarphi}_{2,1}^{\tilde{G}_1}$ is also nonzero and nonpositive. 
In the case of Type II, we have $\spr(G_1(r_1))=\underline{\zeta}_2(r_1)=r_2<\bar{z}_2^*$ and $\bvarphi_2(w)$ has a pole at $w=\underline{\zeta}_2(r_1)$. Hence, if $r_1<\bar{z}_1^*$, we obtain from equation (\ref{eq:varphi2G1_extension1}) and Proposition \ref{pr:limit_varphi2} that 
\begin{align}
\bvarphi_{2,-1}^{G_1}
&=\lim_{z\to r_1} (r_1-z)\,\bvarphi_2(G_1(z)) \cr
&= \begin{pmatrix} \bzero & \cdots & \bzero & \underline{\zeta}_{2,z}(r_1)^{-1}\,c_{pole}^{\bvarphi_2}\, \bu^{C_2}(r_2)\,\bv_{s_0}^{G_1}(r_1) \end{pmatrix} V_1(r_1)^{-1} \cr
&= \underline{\zeta}_{2,z}(r_1)^{-1}\,c_{pole}^{\bvarphi_2}\, \bu^{C_2}(r_2)\,\bv_{s_0}^{G_1}(r_1) \bu_{s_0}^{G_1}(r_1), 
\label{eq:limit_varphisG1_II_1b}
\end{align}
where $\bu^{C_2}(r_2)\,\bv_{s_0}^{G_1}(r_1)>0$; if $r_1=\bar{z}_1^*$, we also obtain from equation (\ref{eq:varphi2G1_extension2}) and Proposition \ref{pr:limit_varphi2} that 
\begin{align}
\tilde{\bvarphi}_{2,-1}^{\tilde{G}_1}
&= \lim_{\tilde{\Delta}_{\bar{z}_1^*}\ni z\to \bar{z}_1^*} (\bar{z}_1^*-z)^{\frac{1}{2}}\,\bvarphi_2(G_1(z)) \cr
&= \begin{pmatrix} \bzero & \cdots & \bzero & (-\tilde{\alpha}_{s_0}^{G_1})^{-1}\,c_{pole}^{\bvarphi_2}\, \bu^{C_2}(r_2)\,\bv_{s_0}^{G_1}(\bar{z}_1^*) \end{pmatrix} V_1(\bar{z}_1^*)^{-1} \cr
&= (-\tilde{\alpha}_{s_0}^{G_1})^{-1}\,c_{pole}^{\bvarphi_2}\, \bu^{C_2}(r_2)\,\bv_{s_0}^{G_1}(\bar{z}_1^*) \bu_{s_0}^{G_1}(\bar{z}_1^*), 
\label{eq:limit_varphisG1_II_1c}
\end{align}
where $\bu^{C_2}(r_2)\,\bv_{s_0}^{G_1}(\bar{z}_1^*) >0$.
\end{proof}

%
%
\subsection{Proof of Proposition 5.10 of Ref.\ \cite{Ozawa18a}}

Denote by $\lambda^{C_1}(z)$ the eigenvalue of $C_1(z,G_1(z))$ corresponding to $\psi_1(z)$ when $z\in[\underline{z}_1^*,\bar{z}_1^*]$ and by $\bu^{C_1}(z)$ and $\bv^{C_1}(z)$ the left and right eigenvectors of $C_1(z,G_1(z))$ with respect to the eigenvalue $\lambda^{C_1}(z)$, satisfying $\bu^{C_1}(z) \bv^{C_1}(z)=1$. 
$\lambda^{C_2}(w)$, $\bu^{C_2}(w)$ and $\bv^{C_2}(w)$ are analogously defined for $C_2(G_2(w),w)$. 

\bigskip\noindent\textbf{Proposition 5.10 of Ref.\ \cite{Ozawa18a}.} 
{\it
For any $z_0\in[\underline{z}_1^*,\bar{z}_1^*]$ such that $\psi_1(z_0)=1$ and for any $k\in\mathbb{Z}_+$, $\bu^{C_1}(z_0)$ and $\bu^{C_1}(z_0) A_{*,1}^{(1)}(z_0) N_1(z_0) R_1(z_0)^k$ are positive. 
Analogously, for any $w_0\in[\underline{z}_2^*,\bar{z}_2^*]$ such that $\psi_2(w_0)=1$ and for any $k\in\mathbb{Z}_+$, $\bu^{C_2}(w_0)$ and $\bu^{C_2}(w_0) A_{1,*}^{(2)}(w_0) N_2(w_0) R_2(w_0)^k$ are also positive.
}
%
%
\begin{proof}
We prove only the first half of the proposition. 
For $z\in[\underline{z}_1^*,\bar{z}_1^*]$, define a nonnegative block tri-diagonal matrix $A^{(1)}_*(z)$ as 
\[
A^{(1)}_*(z) = 
\begin{pmatrix}
A^{(1)}_{*,0}(z) & A^{(1)}_{*,1}(z) & & & \cr
A_{*,-1}(z) & A_{*,0}(z) & A_{*,1}(z) & & \cr
& A_{*,-1}(z) & A_{*,0}(z) & A_{*,1}(z) & \cr
& & \ddots & \ddots & \ddots 
\end{pmatrix}.
\]
Under Assumption \ref{as:Akl_irreducible}, $A^{(1)}_*(z)$ is irreducible and aperiodic. 
Furthermore, for any $z_0\in[\underline{z}_1^*,\bar{z}_1^*]$ such that $\psi_1(z_0)=1$, the invariant measure $\bu_*^{(1)}(z_0)$ satisfying $\bu_*^{(1)}(z_0) A^{(1)}_*(z_0) = \bu_*^{(1)}(z_0)$ is given as follows (see Theorem 3.1 of Ozawa \cite{Ozawa13}): 
\begin{align}
\bu_*^{(1)}(z_0) &= \big( \bu^{C_1}(z_0)\ \ \bu^{C_1}(z_0) A_{*,1}^{(1)}(z_0) N_1(z_0)\ \ \bu^{C_1}(z_0) A_{*,1}^{(1)}(z_0) N_1(z_0) R_1(z_0) \cr
&\qquad\qquad \bu^{C_1}(z_0) A_{*,1}^{(1)}(z_0) N_1(z_0) R_1(z_0)^2\ \ \cdots\  \big).
\end{align}
By Theorem 6.3 of Seneta \cite{Seneta06}, this $\bu_*^{(1)}(z_0)$ is positive and hence the results of the proposition hold. 
\end{proof}

%
%
\subsection{Derivation of the coefficient vectors in Lemma 5.4 of Ref.\ \cite{Ozawa18a}} \label{sec:derivation_lemma5_4}

The vector generating function $\bvarphi_1(z)$ is  defined as $\bvarphi_1(z) = \sum_{i=1}^\infty \bnu_{i,0} z^i$. For notation undefined in this subsection, see Section 5 of Ref.\ \cite{Ozawa18a}. 

\bigskip\noindent\textbf{Lemma 5.4 of Ref.\ \cite{Ozawa18a}.} 
{\it 
(1) If $\psi_1(\bar{z}_1^*)=1$, then the point $z=\bar{z}_1^*$ is a branch point of $\bvarphi_1(z)$ with order one and it is possibly a pole with order one. The Puiseux series for $\bvarphi_1(z)$ at $z=\bar{z}_1^*$ is represented as 
\begin{equation}
\bvarphi_1(z) 
= \tilde{\bvarphi}_1((\bar{z}_1^*-z)^{\frac{1}{2}}) 
= \sum_{k=-1}^\infty \tilde{\bvarphi}_{1,k}^I\,(\bar{z}_1^*-z)^{\frac{k}{2}}, 
\label{eq:varphi1_expansion_Ib}
\end{equation}
where 
\begin{equation}
\tilde{\bvarphi}_{1,-1}^I 
=(-\tilde{\lambda}^{C_1}_\zeta(0))^{-1} \bg_1(\bar{z}_1^*)\bv^{C_1}(\bar{z}_1^*)\,\bu^{C_1}(\bar{z}_1^*). 
\label{eq:tildevarphiI1m1}
\end{equation}

(2) If $\psi_1(\bar{z}_1^*)<1$, then the point $z=\bar{z}_1^*$ is a branch point of $\bvarphi_1(z)$ with order one and its Puiseux series is represented as 
\begin{equation}
\bvarphi_1(z) 
= \tilde{\bvarphi}_1((\bar{z}_1^*-z)^{\frac{1}{2}}) 
= \sum_{k=0}^\infty \tilde{\bvarphi}_{1,k}^I\,(\bar{z}_1^*-z)^{\frac{k}{2}}, 
\label{eq:varphi1_expansion_Ic}
\end{equation}
where $\tilde{\bvarphi}_{1,0}^I=\bvarphi_1(\bar{z}_1^*)$ and 
\begin{align}
\tilde{\bvarphi}_{1,1}^I &= 
\Bigl( \underline{\zeta}_2(\bar{z}_1^*)^{-1} \bvarphi_2(\underline{\zeta}_2(\bar{z}_1^*)) \bigl( A_{*,0}^{(2)}(\bar{z}_1^*) + A_{*,1}^{(2)}(\bar{z}_1^*) ( \underline{\zeta}_2(\bar{z}_1^*) I + G_1(\bar{z}_1^*) ) \bigr) \cr
&\quad + \sum_{k=1}^\infty \bnu_{0,k} \hat{C}_2(\bar{z}_1^*,G_1(\bar{z}_1^*)) \sum_{l=1}^{k-1} \underline{\zeta}_2(\bar{z}_1^*)^{k-l-1} G_1(\bar{z}_1^*)^{l-1}  \cr
&\quad -\sum_{k=1}^\infty \bnu_{0,k} \sum_{l=1}^k \underline{\zeta}_2(\bar{z}_1^*)^{k-l} G_1(\bar{z}_1^*)^{l-1} 
+ \bnu_{0,0} A_{*,1}^{(0)}(\bar{z}_1^*) \cr
&\quad + \bvarphi_1(\bar{z}_1^*) A_{*,1}^{(1)}(\bar{z}_1^*) \Bigr) \tilde{G}_{1,1} \bigl(I-C_1(\bar{z}_1^*,G_1(\bar{z}_1^*)) \bigr)^{-1}. 
\label{eq:tildevarphiI11}
\end{align}
}

\bigskip
\noindent\textit{Derivation of $\tilde{\bvarphi}_{1,-1}^{I}$ and $\tilde{\bvarphi}_{1,1}^{I}$}:
\begin{itemize}
\item[(1)] $\tilde{\bvarphi}_{1,-1}^{I}$ of formula (\ref{eq:tildevarphiI1m1}).\quad 
From Propositions 5.9 and 5.11 of Ref.\ \cite{Ozawa18a}, 
we obtain 
\begin{align*}
\tilde{\bvarphi}_{1,-1}^{I} 
&= \lim_{\tilde{\Delta}_{\bar{z}_1^*}\ni z\to \bar{z}_1^*} (\bar{z}_1^*-z)^{\frac{1}{2}}\,\frac{\bg_1(z)\,\adj(I-C_1(z,G_1(z)))}{\tilde{f}_1(1,(\bar{z}_1^*-z)^{\frac{1}{2}})} 
=\frac{\bg_1(\bar{z}_1^*)\,\bv^{C_1}(\bar{z}_1^*)\,\bu^{C_1}(\bar{z}_1^*)}{-\tilde{\lambda}^{C_1}_\zeta(0)}.
\end{align*}

\item[(2)] $\tilde{\bvarphi}_{1,1}^{I}$ of formula (\ref{eq:tildevarphiI11}).\quad 
From Propositions 5.7 and 5.8 of Ref.\ \cite{Ozawa18a}, 
we obtain
\begin{align*}
\tilde{\bvarphi}_{1,1}^I 
&= \frac{d}{d\zeta}\, \tilde{\bg}_1(\zeta)\,(I-C_1(\bar{z}_1^*-\zeta^2,\tilde{G}_1(\zeta)))^{-1}\,\Big|_{\zeta=0} \cr
&= \bigl(\tilde{\bvarphi}^{\hat{C}_2}_{2,1} -\tilde{\bvarphi}_{2,1}^{\tilde{G}_1} +\bnu_{0,0} A_{*,1}^{(0)}(\bar{z}_1^*)\, \tilde{G}_{1,1} \bigr) (I-C_1(\bar{z}_1^*,G_1(\bar{z}_1^*)))^{-1} \cr
&\qquad + \bg_1(\bar{z}_1^*) (I-C_1(\bar{z}_1^*,G_1(\bar{z}_1^*)))^{-1} A_{*,1}^{(1)}(\bar{z}_1^*)\, \tilde{G}_{1,1}\, (I-C_1(\bar{z}_1^*,G_1(\bar{z}_1^*)))^{-1}, 
\end{align*}
and this leads us to formula (\ref{eq:tildevarphiI11}). 
\end{itemize}

%
%
\subsection{Derivation of the coefficient vectors in Lemma 5.5 of Ref.\ \cite{Ozawa18a}} \label{sec:derivation_lemma5_5}

For notation undefined in this subsection, see Section 5 of Ref.\ \cite{Ozawa18a}. 

\bigskip
\noindent\textbf{Lemma 5.5 of Ref.\ \cite{Ozawa18a}.} 
{\it 
(1) If $\eta_2^{(c)}<\theta_2^{(c)}$, then the point $z=r_1$ is a pole of $\bvarphi_1(z)$ with order one and its Laurent series is represented as 
\begin{equation}
\bvarphi_1(z) 
= \sum_{k=-1}^\infty \bvarphi_{1,k}^{II}\,(r_1-z)^k, 
\label{eq:varphi1_expansion_IIa}
\end{equation}
where 
\begin{align}
\bvarphi_{1,-1}^{II} 
&= \frac{c_{pole}^{\bvarphi_2} \bu^{C_2}(r_2) A_{1,*}^{(2)}(r_2) N_2(r_2) \bv^{R_2}(r_2)}{\underline{\zeta}_{2,z}(r_1)}\, \bu_{s_0}^{G_1}(r_1) \bigl( I-C_1(r_1,G_1(r_1)) \bigr)^{-1} > \bzero^\top 
\label{eq:varphiII1m1}
\end{align}
and $\bv^{R_2}(r_2)$ is the right eigenvector of $R_2(r_2)$ with respect to the eigenvalue $r_1^{-1}$.

(2) If $\eta_2^{(c)}=\theta_2^{(c)}$ and $\psi_1(\bar{z}_1^*)>1$, then the point $z=r_1$ is a pole of $\bvarphi_1(z)$ with order two and its Laurent series is represented as 
\begin{equation}
\bvarphi_1(z) 
= \sum_{k=-2}^\infty \bvarphi_{1,k}^{II}\,(r_1-z)^k, 
\label{eq:varphi1_expansion_IIb}
\end{equation}
where 
\begin{align}
\bvarphi_{1,-2}^{II} &=
\frac{c_{pole}^{\bvarphi_2} \bu^{C_2}(r_2) A_{1,*}^{(2)}(r_2) N_2(r_2) \bv^{R_2}(r_2) \bu_{s_0}^{G_1}(r_1) \bv^{C_1}(r_1)}{\underline{\zeta}_{2,z}(r_1)\, \psi_{1,z}(r_1)}\, \bu^{C_1}(r_1) > \bzero^\top. 
\label{eq:varphiII1m2}
\end{align}

(3) If $\eta_2^{(c)}=\theta_2^{(c)}$ and $\psi_1(\bar{z}_1^*)=1$, then the point $z=\bar{z}_1^*$ is a branch point of $\bvarphi_1(z)$ with order one and it is also a pole with order two. The Puiseux series for $\bvarphi_1(z)$ at $z=\bar{z}_1^*$ is represented as 
\begin{equation}
\bvarphi_1(z) 
= \tilde{\bvarphi}_1((\bar{z}_1^*-z)^{\frac{1}{2}}) 
= \sum_{k=-2}^\infty \tilde{\bvarphi}_{1,k}^{II}\,(\bar{z}_1^*-z)^{\frac{k}{2}}, 
\label{eq:varphi1_expansion_IIc}
\end{equation}
where 
\begin{align}
\tilde{\bvarphi}_{1,-2}^{II} &=
\frac{c_{pole}^{\bvarphi_2} \bu^{C_2}(r_2) A_{1,*}^{(2)}(r_2) N_2(r_2) \bv^{R_2}(r_2) \bu_{s_0}^{G_1}(\bar{z}_1^*) \bv^{C_1}(\bar{z}_1^*)}{\tilde{\alpha}_{s_0,1}^{G_1}\, \tilde{\lambda}_\zeta^{C_1}(0)}\, \bu^{C_1}(\bar{z}_1^*) > \bzero^\top. 
\label{eq:tildevarphiII1m2}
\end{align}

(4) If $\eta_2^{(c)}=\theta_2^{(c)}$ and $\psi_1(\bar{z}_1^*)<1$, then the point $z=\bar{z}_1^*$ is a branch point of $\bvarphi_1(z)$ with order one and it is also a pole with order one. The Puiseux series for $\bvarphi_1(z)$ at $z=\bar{z}_1^*$ is represented as 
\begin{equation}
\bvarphi_1(z) 
= \tilde{\bvarphi}_1((\bar{z}_1^*-z)^{\frac{1}{2}}) 
= \sum_{k=-1}^\infty \tilde{\bvarphi}_{1,k}^{II}\,(\bar{z}_1^*-z)^{\frac{k}{2}}, 
\label{eq:varphi1_expansion_IId}
\end{equation}
where 
\begin{align}
\tilde{\bvarphi}_{1,-1}^{II} 
&= \frac{c_{pole}^{\bvarphi_2} \bu^{C_2}(r_2) A_{1,*}^{(2)}(r_2) N_2(r_2) \bv^{R_2}(r_2)}{-\tilde{\alpha}_{s_0,1}^{G_1}}\, \bu_{s_0}^{G_1}(\bar{z}_1^*) \bigl( I-C_1(\bar{z}_1^*,G_1(\bar{z}_1^*)) \bigr)^{-1} > \bzero^\top. 
\label{eq:tildevarphiII1m1}
\end{align}
}

%
%
\bigskip
\noindent\textit{Derivation of $\bvarphi_{1,-1}^{II}$, $\bvarphi_{1,-2}^{II}$, $\tilde{\bvarphi}_{1,-2}^{II}$ and $\tilde{\bvarphi}_{1,-1}^{II}$}:

\medskip
Before deriving expressions for the coefficient vectors, we give the following proposition. 
\begin{proposition} \label{pr:vR2}
In the case of Type II, we have $\spr(G_1(r_1))=r_2$, $\spr(R_2(r_2))=r_1^{-1}$ and 
\begin{equation}
\bv_{s_0}^{G_1}(r_1) = (r_1 I-G_2(r_2))^{-1} N_2(r_2)\, \bv^{R_2}(r_2). 
\label{eq:vR2}
\end{equation}
%
\end{proposition}
\begin{proof}
In the case of Type II, we have $\spr(G_1(r_1))=\underline{\zeta}_2(r_1)=r_2$, $\spr(R_1(r_1))=\bar{\zeta}_2(r_1)^{-1} \le r_2^{-1}$, $\spr(G_2(r_2))=\underline{\zeta}_1(r_2)<r_1$ and $\spr(R_2(r_2))=\bar{\zeta}_1(r_2)^{-1}=r_1^{-1}$. 
Furthermore, by Lemma 4.1 and Remark 4.1 of Ref.\ \cite{Ozawa18a}, 
we have 
\begin{align}
I-C(r_1,r_2)
&= \bigl(r_2^{-1}I-R_1(r_1)\bigr) \bigl(I-H_1(r_1)\bigr) \bigl(r_2 I-G_1(r_1)\bigr) \cr
&= \bigl(r_1^{-1}I-R_2(r_2)\bigr) \bigl(I-H_2(r_2)\bigr) \bigl(r_1 I-G_2(r_2)\bigr). 
\end{align}
Multiplying both the sides of this equation by $\bv_{s_0}^{G_1}(r_1)$ from the right, we obtain 
\begin{equation}
\bigl(r_1^{-1}I-R_2(r_2)\bigr) \bigl(I-H_2(r_2)\bigr) \bigl(r_1 I-G_2(r_2)\bigr) \bv_{s_0}^{G_1}(r_1) = \bzero.
\end{equation}
Since both $\bigl(I-H_2(r_2)\bigr)$ and $\bigl(r_1 I-G_2(r_2)\bigr)$ are nonsingular, we obtain 
\[
\bv^{R_2}(r_2) = \bigl(I-H_2(r_2)\bigr) \bigl(r_1 I-G_2(r_2)\bigr) \bv_{s_0}^{G_1}(r_1)\ne \bzero, 
\]
and this leads us to expression (\ref{eq:vR2}), where we use the fact that $N_2(r_2)=(I-H_2(r_2))^{-1}$. 
\end{proof}

\bigskip
\begin{itemize}
\item[(1)] $\bvarphi_{1,-1}^{II}$ of formula (\ref{eq:varphiII1m1}).\quad 
Note that we have 
\begin{align}
\bu^{C_2}(r_2) (C_2(r_1,r_2)-I) 
&= \bu^{C_2}(r_2) (C_2(r_1,r_2)-C_2(G_2(r_2),r_2)) \cr
&= \bu^{C_2}(r_2) A_{1,*}^{(2)}(r_2)(r_1 I-G_2(r_2)). 
\label{eq:uC2_C2_G2}
\end{align}
Hence, from Propositions 5.7 and 5.8 of Ref.\ \cite{Ozawa18a}, 
we obtain
\begin{align*}
\bvarphi_{1,-1}^{II} 
&= \lim_{z\to r_1} (r_1-z)\, \bg_1(z)\,(I-C_1(z,G_1(z))^{-1} \cr
&= \bigl(\bvarphi^{\hat{C}_2}_{2,-1} -\bvarphi_{2,-1}^{\tilde{G}_1}\bigr) (I-C_1(r_1,G_1(r_1)))^{-1} \cr
&= \frac{c_{pole}^{\bvarphi_2} \bu^{C_2}(r_2) A_{1,*}^{(2)}(r_2)(r_1 I-G_2(r_2)) \bv_{s_0}^{G_1}(r_1) \bu_{s_0}^{G_1}(r_1) ( I-C_1(r_1,G_1(r_1)) )^{-1}}{\underline{\zeta}_{2,z}(r_1)}.
\end{align*}
This and Proposition \ref{pr:vR2} lead us to formula (\ref{eq:varphiII1m1}). 

\item[(2)] $\bvarphi_{1,-2}^{II}$ of formula (\ref{eq:varphiII1m2}).\quad 
From Propositions 5.3, 5.4, 5.7 and 5.8 of Ref.\ \cite{Ozawa18a}, 
we obtain
\begin{align*}
\bvarphi_{1,-2}^{II} 
&= \lim_{z\to r_1} (r_1-z)^2\, \frac{\bg_1(z)\,\adj (I-C_1(z,G_1(z)))}{f_1(1,z)} \cr
&= \frac{ (\bvarphi^{\hat{C}_2}_{2,-1} -\bvarphi_{2,-1}^{\tilde{G}_1} )\, \bv^{C_1}(r_1) \bu^{C_1}(r_1)}{\psi_{1,z}(r_1)} \cr
&= \frac{c_{pole}^{\bvarphi_2} \bu^{C_2}(r_2) A_{1,*}^{(2)}(r_2)(r_1 I-G_2(r_2)) \bv_{s_0}^{G_1}(r_1)\bu_{s_0}^{G_1}(r_1) \bv^{C_1}(r_1) \bu^{C_1}(r_1)}{\underline{\zeta}_{2,z}(r_1) \psi_{1,z}(r_1)}.
\end{align*}
This and Proposition \ref{pr:vR2} lead us to formula (\ref{eq:varphiII1m2}). 

\item[(3)] $\tilde{\bvarphi}_{1,-2}^{II}$ of formula (\ref{eq:tildevarphiII1m2}).\quad
From Propositions 5.7, 5.8, 5.9 and 5.11 of Ref.\ \cite{Ozawa18a}, 
we obtain
\begin{align*}
\tilde{\bvarphi}_{1,-2}^{II} 
&= \lim_{\tilde{\Delta}_{\bar{z}_1^*}\ni z\to \bar{z}_1^*} (\bar{z}_1^*-z)\, \frac{\tilde{\bg}_1((\bar{z}_1^*-z)^{\frac{1}{2}})\,\adj\big(I-C_1(z,\tilde{G}_1((\bar{z}_1^*-z)^{\frac{1}{2}}))\big)}{\tilde{f}_1(1,(\bar{z}_1^*-z)^{\frac{1}{2}})} \cr
&= \frac{(\tilde{\bvarphi}^{\hat{C}_2}_{2,-1} -\tilde{\bvarphi}_{2,-1}^{\tilde{G}_1} )\,\bv^{C_1}(\bar{z}_1^*) \bu^{C_1}(\bar{z}_1^*)}{-\tilde{\lambda}_\zeta^{C_1}(0)} \cr
&= \frac{c_{pole}^{\bvarphi_2} \bu^{C_2}(r_2) A_{1,*}^{(2)}(r_2)(\bar{z}_1^* I-G_2(r_2)) \bv_{s_0}^{G_1}(\bar{z}_1^*)\bu_{s_0}^{G_1}(\bar{z}_1^*)\bv^{C_1}(\bar{z}_1^*) \bu^{C_1}(\bar{z}_1^*)}{\tilde{\alpha}_{s_0,1}^{G_1} \tilde{\lambda}_\zeta^{C_1}(0)}.
\end{align*}
This and Proposition \ref{pr:vR2} lead us to formula (\ref{eq:tildevarphiII1m2}). 

\item[(4)] $\tilde{\bvarphi}_{1,-1}^{II}$ of formula (\ref{eq:tildevarphiII1m1}).\quad
From Propositions 5.7 and 5.8 of Ref.\ \cite{Ozawa18a}, 
we obtain
\begin{align*}
\tilde{\bvarphi}_{1,-1}^{II} 
&= \lim_{\tilde{\Delta}_{\bar{z}_1^*}\ni z\to \bar{z}_1^*} (\bar{z}_1^*-z)^{\frac{1}{2}}\, \tilde{\bg}_1((\bar{z}_1^*-z)^{\frac{1}{2}})\bigl(I-C_1(z,\tilde{G}_1((\bar{z}_1^*-z)^{\frac{1}{2}}))\bigr)^{-1} \cr
&= (\tilde{\bvarphi}^{\hat{C}_2}_{2,-1} -\tilde{\bvarphi}_{2,-1}^{\tilde{G}_1} )\, (I-C_1(\bar{z}_1^*,G_1(\bar{z}_1^*)))^{-1} \cr
&= \frac{c_{pole}^{\bvarphi_2} \bu^{C_2}(r_2) A_{1,*}^{(2)}(r_2)(\bar{z}_1^* I-G_2(r_2)) \bv_{s_0}^{G_1}(\bar{z}_1^*) \bu_{s_0}^{G_1}(\bar{z}_1^*) (I-C_1(\bar{z}_1^*,G_1(\bar{z}_1^*)))^{-1}}{-\tilde{\alpha}_{s_0,1}^{G_1} }.
\end{align*}
This and Proposition \ref{pr:vR2} lead us to formula (\ref{eq:tildevarphiII1m1}). 
\end{itemize}

%
%

\small
\renewcommand{\baselinestretch}{0.9} 

\end{document}